\documentclass[12pt,a4paper]{amsart}
\usepackage{amssymb,eucal}
\usepackage[all,cmtip]{xy}

\calclayout

\showhyphens{pseu-do-func-tors pseu-do-func-tor
pseu-do-lim-its pseu-do-limit pseu-do-nat-u-ral}

\newcommand{\+}{\protect\nobreakdash-}
\renewcommand{\:}{\colon}

\newcommand{\rarrow}{\longrightarrow}
\newcommand{\larrow}{\longleftarrow}

\newcommand{\bu}{{\text{\smaller\smaller$\scriptstyle\bullet$}}}

\newcommand{\lrarrow}{\mskip.5\thinmuskip\relbar\joinrel\relbar\joinrel
 \rightarrow\mskip.5\thinmuskip\relax}
\newcommand{\llarrow}{\mskip.5\thinmuskip\leftarrow\joinrel\relbar
 \joinrel\relbar\mskip.5\thinmuskip\relax}

\DeclareMathOperator{\Hom}{Hom}
\DeclareMathOperator{\Gen}{Gen}
\DeclareMathOperator{\Rel}{Rel}

\newcommand{\sA}{\mathsf A}
\newcommand{\sB}{\mathsf B}
\newcommand{\sC}{\mathsf C}
\newcommand{\sD}{\mathsf D}
\newcommand{\sE}{\mathsf E}
\newcommand{\sG}{\mathsf G}
\newcommand{\sH}{\mathsf H}
\newcommand{\sJ}{\mathsf J}
\newcommand{\sK}{\mathsf K}
\newcommand{\sL}{\mathsf L}
\newcommand{\sM}{\mathsf M}
\newcommand{\sS}{\mathsf S}
\newcommand{\sT}{\mathsf T}

\newcommand{\Sets}{\mathsf{Sets}}
\newcommand{\Fun}{\mathsf{Fun}}
\newcommand{\Cat}{\mathsf{Cat}}
\newcommand{\Psd}{\mathsf{Psd}}
\newcommand{\TwoCat}{\mathbf{Cat}}
\newcommand{\TwoCAT}{\mathbf{CAT}}
\newcommand{\ACC}{\mathbf{ACC}}

\newcommand{\cF}{\mathcal F}
\newcommand{\cO}{\mathcal O}
\newcommand{\cC}{\mathcal C}

\newcommand{\epi}{\mathsf{epi}}
\newcommand{\mono}{\mathsf{mono}}
\newcommand{\srj}{\mathsf{surj}}
\newcommand{\ses}{\mathsf{ses}}
\newcommand{\fl}{\mathsf{fl}}
\newcommand{\dfl}{{\operatorname{\mathsf{-fl}}}}

\newcommand{\Modl}{{\operatorname{\mathsf{--Mod}}}}
\newcommand{\Comodl}{{\operatorname{\mathsf{--Comod}}}}
\newcommand{\Qcoh}{{\operatorname{\mathsf{--Qcoh}}}}

\newcommand{\id}{\mathrm{id}}

\newcommand{\boZ}{\mathbb Z}

\newcommand{\Section}[1]{\bigskip\section{#1}\medskip}
\theoremstyle{plain}
\newtheorem{thm}{Theorem}[section]

\newtheorem{lem}[thm]{Lemma}

\newtheorem{prop}[thm]{Proposition}
\newtheorem{cor}[thm]{Corollary}
\theoremstyle{definition}

\newtheorem{rem}[thm]{Remark}

\begin{document}

\author{Leonid Positselski}

\address{Institute of Mathematics, Czech Academy of Sciences \\
\v Zitn\'a~25, 115~67 Praha~1 \\ Czech Republic}

\email{positselski@math.cas.cz}

\title{Notes on limits of accessible categories}

\begin{abstract}
 Let $\kappa$ be a regular cardinal, $\lambda<\kappa$ be a smaller
infinite cardinal, and $\sK$ be a $\kappa$\+accessible category
where colimits of $\lambda$\+indexed chains exist.
 We show that various category-theoretic constructions applied to $\sK$,
such as the inserter and the equifier, produce $\kappa$\+accessible
categories $\sE$ again, and the most obvious expected description of
the full subcategory of $\kappa$\+presentable objects in $\sE$ in
terms of $\kappa$\+presentable objects in $\sK$ holds true.
 In particular, if $\sC$ is a $\kappa$\+small category, then
the category of functors $\sC\rarrow\sK$ is $\kappa$\+accessible,
and its $\kappa$\+presentable objects are precisely all the functors
from $\sC$ to the $\kappa$\+presentable objects of~$\sK$.
 We proceed to discuss the preservation of $\kappa$\+accessibility
by conical pseudolimits, lax and oplax limits, and
weighted pseudolimits.
 The results of this paper go back to an unpublished 1977~preprint of
Ulmer.
 Our motivation comes from the theory of flat modules and flat
quasi-coherent sheaves.
\end{abstract}

\maketitle

\tableofcontents

\section*{Introduction}
\medskip

\setcounter{subsection}{-1}

\subsection{{}} \label{prelim-informal-subsecn}
 Let $\kappa$~be a cardinal and $\sK$ be a category such that all
the objects of $\sK$ are $\kappa$\+filtered colimits of (suitably
defined) ``objects of small size relative to~$\kappa$\,''.
 Suppose $\sE$ is the category of objects from $\sK$ or collections
of objects from $\sK$ with a certain additional structure and/or
some equations imposed.
 Is every object of $\sE$ a $\kappa$\+filtered colimit of objects whose
underlying objects from $\sK$ have small size relative to~$\kappa$\,?

\subsection{{}}
 To specify the context of the discussion, let $\kappa$~be
a regular cardinal and $\sK$ be a $\kappa$\+accessible category
(in the sense of~\cite[\S2.1]{MP} or~\cite[Chapter~2]{AR}).
 Let $\sC$ be a $\kappa$\+small category, and let $\sE=\Fun(\sC,\sK)$
be the category of functors $\sC\rarrow\sK$.
 Ideally, one may wish to claim that the category $\Fun(\sC,\sK)$ is
$\kappa$\+accessible and its $\kappa$\+presentable objects are precisely
all the functors $\sC\rarrow\sK_{<\kappa}$, where $\sK_{<\kappa}$ is
the full subcategory of $\kappa$\+presentable objects in~$\sK$.
 But is it true?

 The ``ideal'' state of affairs described in the previous paragraph
was claimed as a general result in a 1988 paper~\cite[Lemma~5.1]{Mak}.
 A general outline of a proof of the lemma was presented in~\cite{Mak};
the details were declared to be ``direct calculations'' and omitted.
 A refutation came in the recent preprint~\cite[Theorem~1.3]{Hen}.
 The ideal state of affairs does not hold in general.

\subsection{{}} 
 The assertions of~\cite[Theorem~1.3]{Hen} provide a complete
characterization of all small categories $\sC$ such that the ``ideal''
statement holds \emph{for all $\kappa$\+accessible categories~$\sK$}.
 All such categories $\sC$ are essentially $\kappa$\+small, but being
essentially $\kappa$\+small is \emph{not} enough.
 The category $\sC$ needs to be also \emph{well-founded} in
the sense of the definition in~\cite{Hen}.

 But are there some $\kappa$\+accessible categories $\sK$ that are
better behaved than some other ones, with respect to the question
at hand?
 Another theorem from~\cite{Hen} tells that there are.
 According to~\cite[Theorem~1.2]{Hen}, if the category $\sC$ is
$\kappa$\+small and the category $\sK$ is locally $\kappa$\+presentable
(in the sense of~\cite{GU} or~\cite[Chapter~1]{AR}), then the functor
category $\Fun(\sC,\sK)$ is locally $\kappa$\+presentable and its full
subcategory of $\kappa$\+presentable objects is
$\Fun(\sC,\sK_{<\kappa})\subset\Fun(\sC,\sK)$.

\subsection{{}}
 Are there any better behaved $\kappa$\+accessible categories beyond
the locally $\kappa$\+pre\-sent\-able ones?
 The present paper purports to answer this question by generalizing
the result of~\cite[Theorem~1.2]{Hen}.

 We show that the following much weaker version of local presentability
is sufficient to guarantee the ``ideal state of affairs'': it is
enough to assume existence of an infinite cardinal $\lambda<\kappa$
such that colimits of all $\lambda$\+indexed chains of objects and
morphisms exist in~$\sK$.
 If this is the case and $\sK$ is $\kappa$\+accessible, then for any
$\kappa$\+small category $\sC$ the category $\Fun(\sC,\sK)$ is also
$\kappa$\+accessible, and the $\kappa$\+presentable objects of
$\Fun(\sC,\sK)$ are precisely all the functors
$\sC\rarrow\sK_{<\kappa}$.
 This is the result of our Theorem~\ref{nonadditive-diagram-theorem}.

 Let us mention that the idea of our condition on a category $\sK$
involving a pair of cardinals $\lambda<\kappa$ is certainly not new.
 It appeared in the discussion of \emph{pseudopullbacks}
in~\cite[Proposition~3.1]{CR} and~\cite[Theorem~2.2]{RR} (and our
arguments in this paper bear some similarity to the one in~\cite{CR}).
 The fact that this condition is sufficient for the ``ideal'' result
on accessibility of diagram categories $\Fun(\sC,\sK)$ (our
Theorem~\ref{nonadditive-diagram-theorem}) seems to be if not quite
new, then a ``well-forgotten old'' discovery, however.

\subsection{{}}
 The discussion in Section~\ref{prelim-informal-subsecn} suggests that
we are also interested in other category-theoretic constructions beyond
the categories of functors or diagrams; and indeed we are.

 \emph{Limits of accessible categories} are mentioned in the title of
this paper.
 There are many relevant concepts of limits of categories, the most
general ones being the weighted pseudolimits or weighted
bilimits~\cite[\S5.1]{MP}, \cite{Kel}, \cite{BKPS}.
 All of them can be built from certain elementary building blocks.
 
 We discuss the \emph{Cartesian product} (easy), the \emph{equifier}
(a representative case for our techniques), the \emph{inserter}
(difficult), and the \emph{pseudopullback} (for which our result is
already known in relatively recent literature~\cite{CR,RR}), as well
as the nonadditive and the additive/$k$\+linear diagram categories.
 The pseudopullbacks and the diagram categories are built from
the products, the inserters, and the equifiers.

 In fact, all weighted pseudolimits and weighted bilimits can be built
from products, inserters, and equifiers, up to category
equivalence~\cite{Kel,BKPS}.
 Hence the importance of our detailed discussion of the products,
the inserters, and the equifiers in the general context of limits
of accessible categories.

\subsection{{}}
 In all the settings (with the exception of the trivial case of
the Cartesian products), our results are very similar.
 The main assumptions are that $\kappa$~is a regular cardinal and
$\lambda<\kappa$ is a smaller infinite cardinal (so the case of
finitely accessible categories, $\kappa=\aleph_0$, is excluded).
 The category $\sK$ is assumed to be $\kappa$\+accessible with
colimits of $\lambda$\+indexed chains.
 If this is the case, then the category $\sE$ of (collections of)
objects from $\sK$ with an additional structure satisfying some
equations is also $\kappa$\+accessible (again with colimits of
$\lambda$\+indexed chains), and the $\kappa$\+presentable objects
of $\sE$ are precisely those whose underlying objects are
$\kappa$\+presentable in~$\sK$.

 We do not dare to speculate on what the author of the paper~\cite{Mak}
might have in mind back in 1988, but the proofs of our results seem
to follow the general outline suggested in~\cite[proof of
Lemma~5.1]{Mak}.
 They are, indeed, ``direct calculations'' (which, however, get
complicated at times).

\subsection{{}}
 In fact, our results go back all the way to late 1970s, to
an unpublished 1977~preprint of Ulmer~\cite{Ulm}.
 The very concept and terminology of an \emph{accessible category} was
only introduced by Makkai and Par\'e in their 1989~book~\cite{MP}.
 Accordingly, the exposition in~\cite{Ulm} was written mainly in
the generality of locally presentable categories (which had been
known since the 1971~book of Gabriel and Ulmer~\cite{GU}).

 The main results of~\cite{Ulm} relevant in our context
are~\cite[Theorem~3.8 and Corollary~3.9]{Ulm}.
 These are stated for locally presentable categories, followed by
a remark~\cite[Remark~3.11(II)]{Ulm} explaining that the assertions
are actually valid for some (what we would now call) accessible 
categories.
 This work of Ulmer was subsequently taken up and developed in
the 1984~dissertation of Bird~\cite{Bir}, which was also written
in the generality of locally presentable categories.
 Ulmer's remark~\cite[Remark~3.11(II)]{Ulm} was not taken up, and
apparently remained almost forgotten.

 The topic of limits of accessible categories was studied by
Makkai and Par\'e~\cite[\S5.1]{MP} using methods which seem to be
quite different from those of Ulmer.
 The Limit Theorem of Makkai and Par\'e~\cite[Theorem~5.1.6]{MP}
claimed that all weighted bilimits of accessible categories are
accessible, but offered no cardinality estimate on the accessibility
rank.
 The fact that a tight estimate can be obtained from Ulmer's results
was not realized.
 See our Corollary~\ref{weighted-pseudolimit-corollary}.

 The present author learned about the existence of Ulmer's preprint
from~\cite[paragraph after Pseudopullback Theorem~2.2]{RR}, where
the knowledge about Ulmer's work is attributed to Porst.
 Still, no traces of such knowledge can be found in Porst's own
earlier paper~\cite{Por} (cf.~\cite[Remark~3.2]{Pflcc} and~\cite{Pcor}).
 I~only got hold of my copy of Ulmer's preprint after the first version
of the present paper, with my own detailed proofs of the main results,
was already available on the \texttt{arXiv}.

\subsection{{}}
 Let us explain our motivation now.
 In terms of the intended applications, we are primarily interested
in the ``minimal cardinality'' case $\kappa=\aleph_1$ and
$\lambda=\aleph_0$.
 The examples we care about arise from flat modules over rings,
flat quasi-coherent sheaves over schemes, flat comodules, and
flat contramodules.

 It is shown in the preprint~\cite[Theorem~2.4]{PS6} that
the category $X\Qcoh_\fl$ of flat quasi-coherent sheaves on
a quasi-compact quasi-separated scheme $X$ is $\aleph_1$\+accessible.
 More genenerally, the same holds for any countably quasi-compact,
countably quasi-separated scheme~\cite[Theorem~3.5]{PS6}.
 The $\aleph_1$\+presentable objects of $X\Qcoh_\fl$ are the locally
countably presentable flat quasi-coherent sheaves, i.~e.,
the quasi-coherent sheaves $\cF$ on $X$ such that the $\cO_X(U)$\+module
$\cF(U)$ is flat and countably presented for all affine open subschemes
$U\subset X$ (equivalently, for the affine open subschemes $U_\alpha$
appearing in some fixed affine open covering
$X=\bigcup_\alpha U_\alpha$ of the scheme~$X$).
 Obviously, all directed colimits, and in particular directed colimits
of $\aleph_0$\+indexed chains, exist in $\sK=X\Qcoh_\fl$.
 So the results of this paper are applicable to this category.

 The results of~\cite{PS6} were extended to certain noncommuative
stacks and noncommutative ind-affine ind-schemes in
the preprint~\cite{Pflcc}.
 Specifically, let $\cC$ be a (coassociative, counital) coring over
a noncommutative ring~$A$.
 According to~\cite[Theorem~3.1]{Pflcc}, the category of $A$\+flat
left $\cC$\+comodules $\cC\Comodl_{A\dfl}$ is $\aleph_1$\+accessible.
 The $\aleph_1$\+presentable objects of of $\cC\Comodl_{A\dfl}$ are
the $A$\+countably presentable $A$\+flat left $\cC$\+comodules.
 Once again, it is obvious that all directed colimits exist in
$\cC\Comodl_{A\dfl}$; so the results of the present paper can be
applied.
 There is also a version for flat contramodules over certain
topological rings~\cite[Theorem~10.1]{Pflcc}, where the results of
the present paper are applicable as well.

\subsection{{}}
 Some results about constructing $A$\+pure acyclic complexes
of $A$\+flat $\cC$\+comod\-ules as $\aleph_1$\+filtered colimits of
$A$\+pure acyclic complexes of $A$\+countably presentable
$A$\+flat $\cC$\+comodules are discussed in~\cite[Section~4]{Pflcc}.
 A contramodule version can be found in~\cite[Section~11]{Pflcc}.
 The techniques developed in the present paper are used throughout
the current (new) versions of the papers~\cite{PS6} and~\cite{Pflcc}.
 The same methods are also used in the preprint~\cite{Pres}, where
accessibility of categories of modules of finite flat dimension and
two-sided/F\+totally acyclic flat resolutions is discussed, and
in the preprint~\cite{Pcor}, where we discuss local presentability
and accessibility ranks of the categories of corings and coalgebras
over rings.

 In the present paper, we do not go into any details on sheaves,
comodules, or contramodules, restricting ourselves to ``toy examples''
of diagrams and complexes of modules over a noncommutative ring~$R$.
 It is easy to see that the category of flat left $R$\+modules
$R\Modl_\fl$ is $\kappa$\+accessible for any regular cardinal~$\kappa$;
the $\kappa$\+presentable objects of $R\Modl_\fl$ are those
flat $R$\+modules that are $\kappa$\+presentable in the category of
arbiratry $R$\+modules $R\Modl$.
 Applying the results of this paper, we obtain descriptions of
diagrams of flat modules and pure acyclic complexes of flat modules
as directed colimits (recovering, in particular, a weaker version of
a result from the papers~\cite{EG,Neem} with very general
category-theoretic methods).

\subsection*{Acknowledgement}
 I~am grateful to Ji\v r\'{\i} Rosick\'y for bringing
the preprint~\cite{Hen} to my attention, as well as for providing
me with a copy of Ulmer's preprint~\cite{Ulm}.
 Simon Henry suggested to include some general, abstract material on
lax limits and weighted pseudolimits; I~followed this suggestion by
writing Sections~\ref{2cat-prelim-secn}, \ref{conical-secn},
and~\ref{weighted-secn}.
 The idea to include Corollary~\ref{weighted-pseudolimit-corollary}
is also due to Simon.
 I~want to thank Jan \v St\!'ov\'\i\v cek for helpful conversations.
 The author is supported by the GA\v CR project 23-05148S and
the Czech Academy of Sciences (RVO~67985840).

\Section{Preliminaries} \label{preliminaries-secn}

 We use the book~\cite{AR} as the main background reference source on
the foundations of the theory of accessible categories.

 Let $\kappa$~be a regular cardinal.
 We refer to~\cite[Definition~1.4, Theorem~1.5, Definition~1.13(1),
and Remark~1.21]{AR} for the discussion of \emph{$\kappa$\+directed
posets} vs.\ \emph{$\kappa$\+filtered categories} and, accordingly,
\emph{$\kappa$\+directed} vs.\ \emph{$\kappa$\+filtered diagrams}
and their colimits.

 Let $\sK$ be a category in which all $\kappa$\+directed (equivalently,
$\kappa$\+filtered) colimits exist.
 An object $S\in\sK$ is said to be
\emph{$\kappa$\+presentable}~\cite[Definitions~1.1 and~1.13(2)]{AR}
if the functor $\Hom_\sK(S,{-})\:\sK\rarrow\Sets$ preserves
$\kappa$\+directed colimits.
 We will denote by $\sK_{<\kappa}\subset\sK$ the full subcategory
of $\kappa$\+presentable objects in~$\sK$.

 A category $\sK$ with $\kappa$\+directed colimits is called
\emph{$\kappa$\+accessible}~\cite[Definition~2.1]{AR} if there is
a set $\sS$ of $\kappa$\+presentable objects in $\sK$ such that every
object of $\sK$ is a $\kappa$\+directed colimit of objects from~$\sS$.
 In any $\kappa$\+accessible category, there is only a set of
isomorphism classes of $\kappa$\+presentable objects; in fact,
the $\kappa$\+presentable objects of $\sK$ are precisely the retracts
of the objects from~$\sS$ \,\cite[Remarks~1.9 and~2.2(4)]{AR}.

 Let $\sK$ be a category and $\sS\subset\sK$ be a set of objects.
 For any object $K\in\sK$,
the \emph{canonical diagram}~\cite[Definition~0.4]{AR}
of morphisms from objects from $\sS$ into $K$ is indexed by the small
indexing category $\Delta=\Delta_{\sS,K}$ whose objects $v\in\Delta$
are morphisms $v\:D_v\rarrow K$ into $K$ from objects $D_v\in\sS$.
 A morphism $a\:v\rarrow w$ in $\Delta$ is a morphism $a\:D_v\rarrow
D_w$ in $\sK$ making the triangular diagram $D_v\rarrow D_w\rarrow K$
commutative in~$\sK$.
 The canonical diagram $D=D_{\sS,K}\:\Delta\rarrow K$ takes an object
$v\in\Delta$ to the object $D_v\in\sK$, and acts on the morphisms in
the obvious way.

\begin{lem} \label{accessible-canonical-colimit}
 Let\/ $\sK$ be a $\kappa$\+accessible category and\/ $\sS$ be
a set of representatives of isomorphism classes of $\kappa$\+presentable
objects in\/~$\sK$.
 Then, for every object $K\in\sK$, the canonical diagram $D=D_{\sS,K}$
of morphisms from objects from $\sS$ into $K$ (or in other words,
its indexing category $\Delta=\Delta_{\sS,K}$) is $\kappa$\+filtered.
 The natural morphism\/ $\varinjlim_{v\in\Delta}D_v\rarrow K$ is
an isomorphism in\/~$\sK$.
\end{lem}

\begin{proof}
 This is~\cite[Definition~1.23 and Proposition~2.8(i\+-ii)]{AR}.
\end{proof}

 Let $\sK$ be a category with $\kappa$\+directed colimits and
$\sA\subset\sK$ be a class of objects (full subcategory).
 Then we denote by $\varinjlim_{(\kappa)}\sA\subset\sK$ the class of
all objects of $\sK$ that can be obtained as $\kappa$\+directed
colimits of objects from~$\sA$.

 The following proposition is also essentially well-known.
 In the particular case of finitely accessible ($\kappa=\aleph_0$)
additive categories, it was discussed in~\cite[Proposition~2.1]{Len},
\cite[Section~4.1]{CB}, and~\cite[Proposition~5.11]{Kra}.
(The terminology ``finitely presented categories'' was used
in~\cite{CB,Kra} for what are called finitely accessible categories
in~\cite{AR}.)

\begin{prop} \label{accessible-subcategory}
 Let\/ $\sK$ be a $\kappa$\+accessible category and\/
$\sS\subset\sK_{<\kappa}$ be a set of $\kappa$\+presentable objects
in\/~$\sK$.
 Then the full subcategory\/ $\varinjlim_{(\kappa)}\sS\subset\sK$ is
closed under $\kappa$\+directed colimits in\/~$\sK$.
 The category\/ $\varinjlim_{(\kappa)}\sS$ is $\kappa$\+accessible;
the full subcategory of all $\kappa$\+presentable objects of\/
$\varinjlim_{(\kappa)}\sS$ consists of all the retracts of objects
from\/ $\sS$ in\/~$\sK$.
 An object $E\in\sK$ belongs to\/ $\varinjlim_{(\kappa)}\sS$ if
and only if, for every $\kappa$\+presentable object $T\in\sK_{<\kappa}$,
every morphism $T\rarrow E$ in\/ $\sK$ factorizes through an object
from\/~$\sS$.
\end{prop}

\begin{proof}
 The key assertion is that if an object $E\in\sK$ has the property that
every morphism $T\rarrow K$ into $K$ from an object $T\in\sK_{<\kappa}$
factorizes through some object from $\sS$, then
$E\in\varinjlim_{(\kappa)}\sS$.
 (All the other assertions follow easily from this one.)

 Indeed, let $\sT$ denote a representative set of $\kappa$\+presentable
objects in~$\sK$.
 Consider the canonical diagram $C\:\Delta_\sS\rarrow\sK$ of
morphisms into $E$ from objects of $\sS$ and the canonical diagram
$D\:\Delta_\sT\rarrow\sK$ of morphisms into $E$ from objects of~$\sT$.
 Then we have $E=\varinjlim_{w\in\Delta_\sT}D_w$ by
Lemma~\ref{accessible-canonical-colimit}, and we need to show that
$E=\varinjlim_{v\in\Delta_\sS}C_v$.
 So it remains to check that the natural functor between the index
categories $\delta\:\Delta_\sS\rarrow\Delta_\sT$ is cofinal in
the sense of~\cite[Section~0.11]{AR}.

 Let $w\:D_w\rarrow E$ be an object of~$\Delta_\sT$.
 Then $D_w\in\sT$, and by assumption the morphism~$v$ factorizes
as $D_w\overset a\rarrow S\overset v\rarrow E$ with $S\in\sS$.
 So $v\:C_v=S\rarrow E$ is an object of $\Delta_\sS$, and we have
a morphism $a\:w\rarrow\delta(v)$ in~$\Delta_\sT$.
 This proves condition~(a) from~\cite[Section~0.11]{AR}.
 Since the category $\Delta_\sT$ is $\kappa$\+filtered and
the functor~$\delta$ is fully faithful, condition~(b) follows
automatically.
\end{proof}

 Any cardinal~$\lambda$ can be considered as a totally ordered set,
which is a particular case of a poset; and any poset $I$ can be
viewed as a category (with the elements of $I$ being the objects,
and a unique morphism $i\rarrow j$ for every pair of objects
$i\le j\in I$).
 A \emph{$\lambda$\+indexed chain} (of objects and morphisms) in
a category $\sK$ is a functor $\lambda\rarrow\sK$, where $\lambda$~is
viewed as a category as explained above.

\Section{Product} \label{product-secn}

 The result of this short section is easy and straightforward; it is
only included here for the sake of completeness of the exposition.
 It is essentially a trivial particular case of~\cite[Theorem~1.3]{Hen},
and also the correct particular case of an erroneous (generally speaking)
argument in~\cite[proof of Proposition~2.67]{AR}.

\begin{prop} \label{product-proposition}
 Let $\kappa$~be a regular cardinal and $(\sK_i)_{i\in I}$ be a family
of $\kappa$\+accessible categories.
 Assume that the cardinality of the indexing set $I$ is smaller
than~$\kappa$.
 Then the Cartesian product\/ $\sK=\prod_{i\in I}\sK_i$ is also
a $\kappa$\+accessible category.
 An object $S\in\sK$, \,$S=(S_i\in\sK_i)_{i\in I}$ is
$\kappa$\+presentable in\/ $\sK$ if and only if all its components
$S_i$ are $\kappa$\+presentable in\/~$\sK_i$.
\end{prop}

\begin{proof}
 The condition that the cardinality of $I$ is smaller than~$\kappa$
(which is missing in~\cite[proof of Proposition~2.67]{AR}) is needed
in order to claim that an object $S\in\sK$ is $\kappa$\+presentable
whenever its components $S_i\in\sK_i$ are $\kappa$\+presentable for
all~$i$.
 Essentially, this holds because $\kappa$\+directed colimits commute
with $\kappa$\+small products in the category of sets
(cf.~\cite[Proposition~2.1]{Hen}).
 Once this is established, it remains to observe that every object
of $\sK$ is a $\kappa$\+directed colimit of such objects $S$, just
as~\cite[proof of Proposition~2.67]{AR} tells.
 Indeed, let $K=(K_i)_{i\in I}\in\sK$ be an object and
$(\Xi_i)_{i\in I}$ be nonempty $\kappa$\+filtered categories such that
$K_i=\varinjlim_{\xi_i\in\Xi_i}S_{i,\xi_i}$ in $\sK_i$ for all $i\in I$
with $S_{i,\xi_i}\in(\sK_i)_{<\kappa}$.
 Then $\Xi=\prod_{i\in I}\Xi_i$ is a $\kappa$\+filtered category and
$K=\varinjlim_{\xi\in\Xi}S_\xi$, where $S_\xi=(S_{i,\xi_i})_{i\in I}$
whenever $\xi=(\xi_i)_{i\in I}\in\prod_{i\in I}\Xi_i$.
 One also needs to use the fact that any retract of an object
$S\in\sK$ with $\kappa$\+presentable components $S_i$ is again
an object with $\kappa$\+presentable components.
\end{proof}

\Section{Equifier} \label{equifier-secn}

 Let $\kappa$~be a regular cardinal and $\lambda$ be a smaller
infinite cardinal, i.~e., $\lambda<\kappa$.
 Let $\sK$ and $\sL$ be $\kappa$\+accessible categories in which
all $\lambda$\+indexed chains (of objects and morphisms) have colimits.
 Let $F$, $G\:\sK\rightrightarrows\sL$ be two parallel functors
preserving $\kappa$\+directed colimits and colimits of
$\lambda$\+indexed chains.
 Assume further that the functor $F$ takes $\kappa$\+presentable
objects to $\kappa$\+presentable objects.
 Let $\phi$, $\psi\:F\rightrightarrows G$ be two parallel natural
transformations of functors.

 Let $\sE\subset\sK$ be the full subcategory consisting of all
objects $E\in\sK$ such that $\phi_E=\psi_E$.
 This construction of the category $\sE$ is known as
the \emph{equifier}~\cite[Section~4]{Kel}, \cite[Section~1]{BKPS},
\cite[Lemma~2.76]{AR}.

 The aim of this section is to prove the following theorem going
back to the unpublished preprint~\cite[Theorem~3.8, Corollary~3.9,
and Remark~3.11(II)]{Ulm}.

\begin{thm} \label{equifier-theorem}
 In the assumptions above, the equifier category\/ $\sE$ is
$\kappa$\+accessible.
 The $\kappa$\+presentable objects of\/ $\sE$ are precisely all
the objects of\/ $\sE$ that are $\kappa$\+presentable as objects
of\/~$\sK$.
\end{thm}

 We start with the obvious observations that $\kappa$\+directed colimits
(as well as colimits of $\lambda$\+indexed chains) exist in $\sE$ and
are preserved by the inclusion functor $\sE\rarrow\sK$ (because such
colimits exist in $\sK$ and are preserved by the functor~$F$).
 It follows immediately that any object of $\sE$ that is
$\kappa$\+presentable in $\sK$ is also $\kappa$\+presentable in~$\sE$.
 The proof of the theorem is based the following proposition.

\begin{prop} \label{equifier-proposition}
 Let $E\in\sE$ be an object and $S\in\sK_{<\kappa}$ be
a $\kappa$\+presentable object.
 Then any morphism $S\rarrow E$ in\/ $\sK$ factorizes through
an object $U\in\sE\cap\sK_{<\kappa}$.
\end{prop}

\begin{proof}
 Let $E=\varinjlim_{\xi\in\Xi}T_\xi$ be a representation of the object
$E$ as a $\kappa$\+filtered colimit of $\kappa$\+presentable objects in
the category~$\sK$.
 Then we have $G(E)=\varinjlim_{\xi\in\Xi}G(T_\xi)$ in $\sL$
and $F(S)$, $F(T_\xi)\in\sL_{<\kappa}$.
 There exists an index $\xi_0\in\Xi$ such that the morphism $S\rarrow E$
factorizes through the morphism $T_{\xi_0}\rarrow E$ in~$\sK$.

 Since $E\in\sE$, we have $\phi_E=\psi_E\:F(E)\rarrow G(E)$.
 Hence the two compositions
$$
 \xymatrix{
  F(T_{\xi_0}) \ar@<2pt>[r]^\phi \ar@<-2pt>[r]_\psi
  & G(T_{\xi_0}) \ar[r] & G(E)
 }
$$
are equal to each other in~$\sL$.
 Since $G(E)=\varinjlim_{\xi\in\Xi}G(T_\xi)$ and $F(T_{\xi_0})\in
\sL_{<\kappa}$, it follows that there exists an index $\xi_1\in\Xi$
together with an arrow $\xi_0\rarrow\xi_1$ in $\Xi$ such that
the two compositions
$$
 \xymatrix{
  F(T_{\xi_0}) \ar@<2pt>[r]^\phi \ar@<-2pt>[r]_\psi
  & G(T_{\xi_0}) \ar[r] & G(T_{\xi_1}) 
 }
$$
are equal to each other in~$\sL$.

 Similarly, there exists an index $\xi_2\in\Xi$ together with an arrow
$\xi_1\rarrow\xi_2$ in $\Xi$ such that the two compositions
$$
 \xymatrix{
  F(T_{\xi_1}) \ar@<2pt>[r]^\phi \ar@<-2pt>[r]_\psi
  & G(T_{\xi_1}) \ar[r] & G(T_{\xi_2}) 
 }
$$
are equal to each other, etc.

 Proceeding in this way, we construct a $\lambda$\+indexed chain
of indices $\xi_i\in\Xi$ and arrows $\xi_i\rarrow\xi_j$ in $\Xi$ for
all $0\le i<j<\lambda$ such that, for all ordinals $0\le i<\lambda$,
the two compositions
$$
 \xymatrix{
  F(T_{\xi_i}) \ar@<2pt>[r]^\phi \ar@<-2pt>[r]_\psi
  & G(T_{\xi_i}) \ar[r] & G(T_{\xi_{i+1}}) 
 }
$$
are equal to each other in~$\sL$.
 Specifically, for a limit ordinal $k<\lambda$, we just pick
an index $\xi_k\in\Xi$ and arrows $\xi_i\rarrow\xi_k$ in $\Xi$ for
all $i<k$ making the triangles $\xi_i\rarrow\xi_j\rarrow\xi_k$
commutative in $\Xi$ for all $i<j<k$.
 This can be done, because $k<\kappa$ and the index category $\Xi$
is $\kappa$\+filtered.
 For a successor ordinal $k=i+1<\lambda$, the same argument as above
in this proof provides the desired arrow $\xi_i\rarrow\xi_{i+1}$.

 After the construction is finished, it remains to put
$U=\varinjlim_{i<\lambda}T_{\xi_i}$.
 We have $U\in\sK_{<\kappa}$, since $\lambda<\kappa$ and the class of
all $\kappa$\+presentable objects in a category with $\kappa$\+directed
colimits is closed under those $\kappa$\+small colimits that exist
in the category~\cite[Proposition~1.16]{AR}.
 We also have $\phi_U=\psi_U$ by construction, since
$F(U)=\varinjlim_{i<\lambda}F(T_{\xi_i})$; so $U\in\sE$.
\end{proof}

\begin{proof}[Proof of Theorem~\ref{equifier-theorem}]
 Combine Propositions~\ref{accessible-subcategory}
and~\ref{equifier-proposition}.
\end{proof}

\begin{rem} \label{joint-equifier}
 In applications of Theorem~\ref{equifier-theorem}, one may be
interested in the \emph{joint equifier} of a family of pairs of natural
transformations (cf.~\cite[Remark~2.76]{AR}).
 Let $\sK$ be a $\kappa$\+accessible category and $(\sL_i)_{i\in I}$
be a family of $\kappa$\+accessible categories.
 Let $F_i$, $G_i\:\sK\rightrightarrows\sL_i$ be a family of pairs
of parallel functors, all of them preserving $\kappa$\+directed
colimits and colimits of $\lambda$\+indexed chains.
 Assume further that the functors $F_i$ take $\kappa$\+presentable
objects to $\kappa$\+presentable objects, and that the cardinality
of the indexing set $I$ is smaller than~$\kappa$.
 Let $\phi_i$, $\psi_i\:F_i\rightrightarrows G_i$ be a family of
pairs of parallel natural transformations.

 Consider the full subcategory $\sE\subset\sK$ consisting of all
objects $E\in\sK$ such that $\phi_{i,E}=\psi_{i,E}$ for all
$i\in I$.
 Then the category $\sE$ is $\kappa$\+accessible, and
the $\kappa$\+presentable objects of $\sE$ are precisely all
the objects of $\sE$ that are $\kappa$\+presentable as objects of~$\sK$.
 This assertion can be deduced from
Proposition~\ref{product-proposition} and Theorem~\ref{equifier-theorem}
by passing to the Cartesian product category $\sL=\prod_{i\in I}\sL_i$.
 The family of functors $F_i\:\sK\rarrow\sL_i$ defines a functor
$F\:\sK\rarrow\sL$, the family of functors $G_i\:\sK\rarrow\sL_i$
defines a functor $G\:\sK\rarrow\sL$, and the family of pairs of
natural transformations $\phi_i$, $\psi_i\:F_i\rightrightarrows G_i$
defines a pair of natural transformations $\phi$,
$\psi\:F\rightrightarrows G$.
 It follows from Proposition~\ref{product-proposition} that all
the assumptions of Theorem~\ref{equifier-theorem} are satisfied by
the category $\sL$ and the pair of functors $F$,~$G$.
\end{rem}

\Section{Inserter} \label{inserter-secn}

 As in Section~\ref{equifier-secn}, we consider a regular
cardinal~$\kappa$ and a smaller infinite cardinal $\lambda<\kappa$.
 Let $\sK$ and $\sL$ be $\kappa$\+accessible categories in which
all $\lambda$\+indexed chains have colimits.
 Let $F$, $G\:\sK\rightrightarrows\sL$ be two parallel functors
preserving $\kappa$\+directed colimits and colimits of
$\lambda$\+indexed chains; assume further that the functor $F$ takes
$\kappa$\+presentable objects to $\kappa$\+presentable objects.

 Let $\sE$ be the category of pairs $(K,\phi)$, where
$K\in\sK$ is an object and $\phi\:F(K)\rarrow G(K)$ is
a morphism in~$\sL$.
 This construction of the category $\sE$ is known as
the \emph{inserter}~\cite[Section~4]{Kel}, \cite[Section~1]{BKPS},
\cite[Section~5.1.1]{MP}, \cite[Section~2.71]{AR}.

 The aim of this section is to prove the following theorem, which
also goes back to the unpublished preprint~\cite[Theorem~3.8,
Corollary~3.9, and Remark~3.11(II)]{Ulm}.

\begin{thm} \label{inserter-theorem}
 In the assumptions above, the inserter category\/ $\sE$ is
$\kappa$\+accessible.
 The $\kappa$\+presentable objects of\/ $\sE$ are precisely
all the pairs $(S,\psi)$ where $S$ is a $\kappa$\+presentable
object of\/~$\sK$.
\end{thm}

 We start with the obvious observations that $\kappa$\+directed colimits
(as well as colimits of $\lambda$\+indexed chains) exist in $\sE$ and
are preserved by the forgetful functor $\sE\rarrow\sK$ (because such
colimits exists in $\sK$ and are preserved by the functor~$F$).

 The proof of the theorem is based on three propositions.
 It uses the same idea as the proof of
Theorem~\ref{equifier-theorem} above, but the details are much
more complicated in the case of Theorem~\ref{inserter-theorem}.

\begin{prop} \label{inserter-obvious-inclusion}
 Let $(S,\psi)\in\sE$ be an object such that $S\in\sK_{<\kappa}$.
 Then $(S,\psi)\in\sE_{<\kappa}$.
\end{prop}

\begin{proof}
 The assumptions concerning cardinal~$\lambda$ are not needed for
this proposition.
 Essentially, the assertion holds because $\kappa$\+directed colimits
commute with finite limits in the category of sets
(cf.~\cite[Proposition~2.1]{Hen}).
 To be more specific, it helps to observe that, given an object
$(K,\phi)$ in $\sE$, the set of morphisms $\Hom_\sE((S,\psi),(K,\phi))$
is computed as the equalizer of the natural pair of maps
$$
 \xymatrix{
  \Hom_\sK(S,K) \ar@<2pt>[rr]^-{f\mapsto\phi\circ F(f)}
  \ar@<-2pt>[rr]_-{f\mapsto G(f)\circ\psi} && \Hom_\sL(F(S),G(K)).
 }
$$
 Then one needs to use the assumptions that the functor $G$ preserves
$\kappa$\+directed colimits and the functor $F$ takes
$\kappa$\+presentable objects to $\kappa$\+presentable objects.
\end{proof}

 Denote by $\sE'_{<\kappa}\subset\sE$ the full subcategory formed by
all the pairs $(S,\psi)\in\sE$ with $S\in\sK_{<\kappa}$.
 By Proposition~\ref{inserter-obvious-inclusion}, we have
$\sE'_{<\kappa}\subset\sE_{<\kappa}$.

\begin{prop} \label{inserter-diagram-kappa-filtered}
 Let $E=(K,\phi)\in\sE$ be an object.
 Consider the canonical diagram $C=D_E$ of morphisms into $E$ from
(representatives of isomorphism classes of) objects $B=(S,\psi)\in
\sE'_{<\kappa}$, with the indexing category $\Delta=\Delta_E$.
 Then the indexing category $\Delta$ is $\kappa$\+filtered.
\end{prop}

\begin{prop} \label{inserter-diagram-cofinal}
 In the context of
Proposition~\ref{inserter-diagram-kappa-filtered}, consider also
the canonical diagram $D=D_K$ of morphisms into $K$ from
(representatives of isomorphism classes of) objects $S\in\sK_{<\kappa}$,
with the indexing category~$\Delta_K$.
 Then the natural functor between the indexing categories
$\Delta_E\rarrow\Delta_K$ is cofinal (in the sense
of\/~\cite[Section~0.11]{AR}).
\end{prop}

 The proofs of Propositions~\ref{inserter-diagram-kappa-filtered}
and~\ref{inserter-diagram-cofinal} are based on the following lemma.

\begin{lem} \label{inserter-main-lemma}
 Let $E=(K,\phi)\in\sE$ be an object, let $S$, $T\in\sK_{<\kappa}$
be $\kappa$\+presentable objects, and let $\sigma\:F(S)\rarrow
G(T)$ be a morphism in\/~$\sL$.
 Let $S\rarrow T$ and $T\rarrow K$ be morphisms in\/~$\sK$.
 Assume that the pentagonal diagram
$$
 \xymatrix{
  F(S) \ar[r] \ar[d]_\sigma & F(T) \ar[r] & F(K) \ar[d]^\phi \\
  G(T) \ar[rr] & & G(K) 
 }
$$
is commutative in\/~$\sL$.
 Then there exists an object $B=(U,\psi)\in\sE'_{<\kappa}$ together
with a morphism $(U,\psi)\rarrow(K,\phi)$ in\/ $\sE$ and a morphism
$T\rarrow U$ in\/ $\sK$ such that the pentagonal diagram
$$
 \xymatrix{
  F(S) \ar[r] \ar[d]_\sigma & F(T) \ar[r] & F(U) \ar[d]^\psi \\
  G(T) \ar[rr] & & G(U) 
 }
$$
is commutative in\/ $\sL$ and the triangular diagram
$T\rarrow U\rarrow K$ is commutative in\/~$\sK$.
\end{lem}

\begin{proof}
 Let $K=\varinjlim_{\xi\in\Xi}T_\xi$ be a representation of the object
$K$ as a $\kappa$\+filtered colimit of $\kappa$\+presentable objects in
the category~$\sK$.
 Then we have $G(K)=\varinjlim_{\xi\in\Xi}G(T_\xi)$ in $\sL$
and $F(S)$, $F(T_\xi)\in\sL_{<\kappa}$.
 There exists an index $\xi_0\in\Xi$ such that the morphism $T\rarrow K$
factorizes through the morphism $T_{\xi_0}\rarrow K$ in~$\sK$.
 Then the heptagonal diagram
$$
 \xymatrix{
  F(S) \ar[r] \ar[d]_\sigma & F(T) \ar[r] & F(T_{\xi_0}) \ar[r]
  & F(K) \ar[d]^\phi \\
  G(T) \ar[rr] & & G(T_{\xi_0}) \ar[r] & G(K) 
 }
$$
is commutative in~$\sL$.

 Since $G(K)=\varinjlim_{\xi\in\Xi}G(T_\xi)$ and $F(T_{\xi_0})
\in\sL_{<\kappa}$, there exists an index $\xi_1\in\Xi$ such that
the composition $F(T_{\xi_0})\rarrow F(K)\rarrow G(K)$ factorizes
through the morphism $G(T_{\xi_1})\rarrow G(K)$ in~$\sL$:
$$
 \xymatrix{
  F(T_{\xi_0}) \ar[r] \ar[d]_{\psi_0} & F(K) \ar[d]^\phi \\
  G(T_{\xi_1}) \ar[r] & G(K) 
 }
$$
 Moreover, since $G(K)=\varinjlim_{\xi\in\Xi}G(T_\xi)$ and
$F(S)\in\sL_{<\kappa}$, one can choose the index~$\xi_1$ together
with an arrow $\xi_0\rarrow\xi_1$ in $\Xi$ such that
the hexagonal diagram
$$
 \xymatrix{
  F(S) \ar[r] \ar[d]_\sigma & F(T) \ar[r] & F(T_{\xi_0})
  \ar[d]^{\psi_0} \\
  G(T) \ar[r] & G(T_{\xi_0}) \ar[r] & G(T_{\xi_1}) 
 }
$$
is commutative in~$\sL$.
 Notice that the pentagonal diagram
$$
 \xymatrix{
  F(T_{\xi_0}) \ar[r] \ar[d]_{\psi_0} & F(T_{\xi_1}) \ar[r]
  & F(K) \ar[d]^\phi \\
  G(T_{\xi_1}) \ar[rr] && G(K) 
 }
$$
is also commutative in~$\sL$.

 Hence one can choose an index $\xi_2\in\Xi$ together with an arrrow
$\xi_1\rarrow\xi_2$ in $\Xi$ such that the composition
$F(T_{\xi_1})\rarrow F(K)\rarrow G(K)$ factorizes through
the morphism $G(T_{\xi_2})\rarrow G(K)$:
$$
 \xymatrix{
  F(T_{\xi_1}) \ar[r] \ar[d]_{\psi_1} & F(K) \ar[d]^\phi \\
  G(T_{\xi_2}) \ar[r] & G(K) 
 }
$$
and the square diagram
$$
 \xymatrix{
  F(T_{\xi_0}) \ar[r] \ar[d]_{\psi_0} & F(T_{\xi_1}) \ar[d]^{\psi_1} \\
  G(T_{\xi_1}) \ar[r] & G(T_{\xi_2}) 
 }
$$
is commutative in~$\sL$.
 Then the pentagonal diagram
$$
 \xymatrix{
  F(T_{\xi_1}) \ar[r] \ar[d]_{\psi_1} & F(T_{\xi_2}) \ar[r]
  & F(K) \ar[d]^\phi \\
  G(T_{\xi_2}) \ar[rr] && G(K) 
 }
$$ 
is commutative in~$\sL$.

 Proceeding in this way, we construct a $\lambda$\+indexed chain
of indices $\xi_i\in\Xi$ and arrows $\xi_i\rarrow\xi_j$ in $\Xi$ for
all $0\le i<j<\lambda$ together with morphisms $\psi_i\:F(T_{\xi_i})
\rarrow G(T_{\xi_{i+1}})$ in $\sL$ such that, for all ordinals
$0\le i<\lambda$, the square diagram 
$$
 \xymatrix{
  F(T_{\xi_i}) \ar[r] \ar[d]_{\psi_i} & F(K) \ar[d]^\phi \\
  G(T_{\xi_{i+1}}) \ar[r] & G(K) 
 }
$$
is commutative in $\sL$ and, for all ordinals $0\le i<j<\lambda$,
the square diagram
$$
 \xymatrix{
  F(T_{\xi_i}) \ar[r] \ar[d]_{\psi_i} & F(T_{\xi_j}) \ar[d]^{\psi_j} \\
  G(T_{\xi_{i+1}}) \ar[r] & G(T_{\xi_{j+1}}) 
 }
$$
is commutative in~$\sL$.

 Specifically, similarly to the proof of
Proposition~\ref{equifier-proposition}, for a limit ordinal $k<\lambda$,
we just pick an index $\xi_k\in\Xi$ and arrows $\xi_i\rarrow\xi_k$ in
$\Xi$ for all $i<k$ making the triangles $\xi_i\rarrow\xi_j\rarrow\xi_k$
commutative in $\Xi$ for all $i<j<k$.
 For a successor ordinal $k=j+1<\lambda$, we choose an index
$\xi_{j+1}\in\Xi$ together with an arrow $\xi_j\rarrow\xi_{j+1}$
in $\Xi$ such that the composition $F(T_{\xi_j})\rarrow F(K)\rarrow
G(K)$ factorizes through the morphism $G(T_{\xi_{j+1}})\rarrow G(K)$:
$$
 \xymatrix{
  F(T_{\xi_j}) \ar[r] \ar[d]_{\psi_j} & F(K) \ar[d]^\phi \\
  G(T_{\xi_{j+1}}) \ar[r] & G(K) 
 }
$$
and the square diagram
$$
 \xymatrix{
  F(T_{\xi_i}) \ar[r] \ar[d]_{\psi_i} & F(T_{\xi_j}) \ar[d]^{\psi_j} \\
  G(T_{\xi_{i+1}}) \ar[r] & G(T_{\xi_{j+1}}) 
 }
$$
is commutative in $\sL$ for all $i<j$.
 The latter condition can be satisfied because the pentagonal
diagrams
$$
 \xymatrix{
  F(T_{\xi_i}) \ar[r] \ar[d]_{\psi_i} & F(T_{\xi_j}) \ar[r]
  & F(K) \ar[d]^\phi \\
  G(T_{\xi_{i+1}}) \ar[rr] && G(K) 
 }
$$ 
are commutative in $\sL$ for all $i<j$ and the index category $\Xi$
is $\kappa$\+filtered.

 After the construction is finished, it remains to put
$U=\varinjlim_{i<\lambda}T_{\xi_i}$, and define $\psi\:F(U)\rarrow
G(U)$ to be the colimit of the morphisms $\psi_i\:F(T_{\xi_i})\rarrow
G(T_{\xi_{i+1}})$.
 It is important here that $F(U)=\varinjlim_{i<\lambda}F(T_{\xi_i})$.
 We have $U\in\sK_{<\kappa}$ for the reason explained in
the proof of Proposition~\ref{equifier-proposition}.
\end{proof}

\begin{proof}[Proof of
Proposition~\ref{inserter-diagram-kappa-filtered}]
 Firstly, let $v_a\:(S_a,\psi_a)\rarrow(K,\phi)$ be a family of
morphisms into $(K,\phi)$ from objects $(S_a,\psi_a)\in\sE'_{<\kappa}$,
with the set of indices~$a$ having cardinality smaller than~$\kappa$.
 We need to show that there is a morphism $u\:(T,\tau)\rarrow(K,\phi)$
into $(K,\phi)$ from an object $(T,\tau)\in\sE'_{<\kappa}$ such that
all the morphisms~$v_a$ factorize through~$u$.
 For this purpose, choose a representation
$K=\varinjlim_{\xi\in\Xi}S_\xi$ of the object $K\in\sK$ as
a $\kappa$\+filtered colimit of $\kappa$\+presentable objects
$S_\xi\in\sK_{<\kappa}$.

 Then there exists an index $\xi_0\in\Xi$ such that all the morphisms
$v_a\:S_a\rarrow K$ factorize through the morphism $S_{\xi_0}\rarrow K$
in~$\sK$.
 The hexagonal diagram
$$
 \xymatrix{
  F(S_a) \ar[r] \ar[d]_{\psi_a} & F(S_{\xi_0}) \ar[r] & F(K)
  \ar[d]^\phi \\
  G(S_a) \ar[r] & G(S_{\xi_0}) \ar[r] & G(K) 
 }
$$
is commutative in $\sL$ for all indices~$a$.
 Therefore, one can choose an index $\xi_1\in\Xi$ together with
an arrow $\xi_0\rarrow\xi_1$ in $\Xi$ such that the composition
$F(S_{\xi_0})\rarrow F(K)\rarrow G(K)$ factorizes through
the morphism $G(S_{\xi_1})\rarrow G(K)$:
$$
 \xymatrix{
  F(S_{\xi_0}) \ar[r] \ar[d]_\sigma & F(K) \ar[d]^\phi \\
  G(S_{\xi_1}) \ar[r] & G(K) 
 }
$$
and the pentagonal diagrams
$$
 \xymatrix{
  F(S_a) \ar[rr] \ar[d]_{\psi_a} && F(S_{\xi_0}) \ar[d]^\sigma \\
  G(S_a) \ar[r] & G(S_{\xi_0}) \ar[r] & G(S_{\xi_1})
 }
$$
are commutative in $\sL$ for all~$a$.
 Then the pentagonal diagram
$$
 \xymatrix{
  F(S_{\xi_0}) \ar[r] \ar[d]_\sigma & F(S_{\xi_1}) \ar[r]
  & F(K) \ar[d]^\phi \\
  G(S_{\xi_1}) \ar[rr] && G(K) 
 }
$$
is also commutative in~$\sL$.
 It remains to put $S=S_{\xi_0}$ and $T=S_{\xi_1}$, and use
Lemma~\ref{inserter-main-lemma}.

 Secondly, let $v\:(P,\pi)\rarrow(K,\phi)$ be a morphism into
$(K,\phi)$ from an object $(P,\pi)\in\sE'_{<\kappa}$, and let
$w_a\:(R,\rho)\rarrow(P,\pi)$ be a family of parallel morphisms
into $(P,\pi)$ from an object $(R,\rho)\in\sE'_{<\kappa}$, with
the set of indices~$a$ having cardinality smaller than~$\kappa$.
 Assume that all the morphisms $vw_a\:(R,\rho)\rarrow(K,\phi)$ are
equal to each other.
 We need to show that the morphism $v\:(P,\pi)\rarrow(K,\phi)$ can be
factorized as $(P,\pi)\overset u\rarrow(U,\psi)\rarrow(K,\phi)$ in
such a way that $(U,\psi)\in\sE'_{<\kappa}$ and all the morphisms
$uw_a\:(R,\rho)\rarrow(U,\psi)$ are equal to each other.

 For this purpose, choose a representation
$K=\varinjlim_{\xi\in\Xi}S_\xi$ of the object $K\in\sK$ as
a $\kappa$\+filtered colimit of $\kappa$\+presentable objects
$S_\xi\in\sK_{<\kappa}$.
 Then there exists an index $\xi_0\in\Xi$ such that the morphism
$v\:P\rarrow K$ factorizes through the morphism $S_{\xi_0}\rarrow K$
and all the compositions $R\overset{w_a}\rarrow P\rarrow S_{\xi_0}$
are equal to each other.
 The hexagonal diagram
$$
 \xymatrix{
  F(P) \ar[r] \ar[d]_\pi & F(S_{\xi_0}) \ar[r] & F(K)
  \ar[d]^\phi \\
  G(P) \ar[r] & G(S_{\xi_0}) \ar[r] & G(K) 
 }
$$
is commutative in~$\sL$.
 
 Therefore, one can choose an index $\xi_1\in\Xi$ together with
an arrow $\xi_0\rarrow\xi_1$ in $\Xi$ such that the composition
$F(S_{\xi_0})\rarrow F(K)\rarrow G(K)$ factorizes through
the morphism $G(S_{\xi_1})\rarrow G(K)$:
$$
 \xymatrix{
  F(S_{\xi_0}) \ar[r] \ar[d]_\sigma & F(K) \ar[d]^\phi \\
  G(S_{\xi_1}) \ar[r] & G(K) 
 }
$$
and the pentagonal diagram
$$
 \xymatrix{
  F(P) \ar[rr] \ar[d]_\pi && F(S_{\xi_0}) \ar[d]^\sigma \\
  G(P) \ar[r] & G(S_{\xi_0}) \ar[r] & G(S_{\xi_1})
 }
$$
is commutative in~$\sL$.
 Once again, it remains to put $S=S_{\xi_0}$ and $T=S_{\xi_1}$, and
refer to Lemma~\ref{inserter-main-lemma}.
\end{proof}

\begin{proof}[Proof of
Proposition~\ref{inserter-diagram-cofinal}]
 Firstly, let $P\rarrow K$ be a morphism into $K$ from an object
$P\in\sK_{<\kappa}$.
 We need to show that there exists an object $(U,\psi)\in\sE'_{<\kappa}$
together with a morphism $(U,\psi)\rarrow(K,\phi)$ in $\sE$ and
a morphism $P\rarrow U$ in $\sK$ such that the triangular diagram
$P\rarrow U\rarrow K$ is commutative in~$\sK$.

 For this purpose, choose a representation
$K=\varinjlim_{\xi\in\Xi}T_\xi$ of the object $K\in\sK$ as
a $\kappa$\+filtered colimit of $\kappa$\+presentable objects
$T_\xi\in\sK_{<\kappa}$.
 Then there exists an index $\xi_1\in\Xi$ such that the morphism
$P\rarrow K$ factorizes through the morphism $T_{\xi_1}\rarrow K$
in $\sK$ and the composition $F(P)\rarrow F(K)\rarrow G(K)$ factorizes
through the morphism $G(T_{\xi_1})\rarrow G(K)$ in~$\sL$:
$$
 \xymatrix{
  F(P) \ar[r] \ar[d]_\sigma & F(K) \ar[d]^\phi \\
  G(T_{\xi_1}) \ar[r] & G(K) 
 }
$$
 It remains to put $S=P$ and $T=T_{\xi_1}$, and refer to
Lemma~\ref{inserter-main-lemma}.

 Secondly, let $(R',\rho')$ and $(R'',\rho'')$ be two objects
of $\sE'_{<\kappa}$, let
$$
 (R',\rho')\lrarrow(K,\phi)\llarrow(R'',\rho'')
$$
be two morphisms in $\sE$, and let $R'\larrow P\rarrow R''$ be two
morphisms in $\sK$ such that the square diagram
$$
 \xymatrix{
  & P \ar[ld] \ar[rd] \\
  R' \ar[r] & K & R'' \ar[l]
 }
$$
is commutative in~$\sK$.
 We need to show that there exists an object $(U,\psi)\in\sE_{<\kappa}$
together with two morphisms $(R',\rho')\rarrow(U,\psi)\larrow
(R'',\rho'')$ and a morphism $(U,\psi)\rarrow(K,\phi)$ in $\sE$
such that the two triangular diagrams
$$
 \xymatrix{
  (R',\rho') \ar[r] \ar[rd] & (U,\psi) \ar[d]
  & (R'',\rho'') \ar[l] \ar[ld] \\ & (K,\phi)
 }
$$
are commutative in $\sE$ and the square diagram
$$
 \xymatrix{
  & P \ar[ld] \ar[rd] \\
  R' \ar[r] & U & R'' \ar[l]
 }
$$
is commutative in~$\sK$.

 For this purpose, choose a representation
$K=\varinjlim_{\xi\in\Xi}S_\xi$ of the object $K\in\sK$ as
a $\kappa$\+filtered colimit of $\kappa$\+presentable objects
$S_\xi\in\sK_{<\kappa}$.
 Then there exists an index $\xi_0\in\Xi$ such that both the morphisms
$R'\rarrow K$ and $R''\rarrow K$ factorize through the morphism
$S_{\xi_0}\rarrow K$ in $\sK$ and the square diagram
$$
 \xymatrix{
  & P \ar[ld] \ar[rd] \\
  R' \ar[r] & S_{\xi_0} & R'' \ar[l]
 }
$$
is commutative in~$\sK$.
 So the whole diagram
$$
 \xymatrix{
  & P \ar[ld] \ar[rd] \\
  R' \ar[r] \ar[rd] & S_{\xi_0} \ar[d] & R'' \ar[l] \ar[ld] \\
  & K
 }
$$
is commutative.
 Then the two hexagonal diagrams
$$
 \xymatrix{
  F(R') \ar[r] \ar[d]_{\rho'} & F(S_{\xi_0}) \ar[r] & F(K) \ar[d]^\phi
  & F(S_{\xi_0}) \ar[l] & F(R'') \ar[l] \ar[d]^{\rho''} \\
  G(R') \ar[r] & G(S_{\xi_0}) \ar[r] & G(K)
  & G(S_{\xi_0}) \ar[l] & G(R'') \ar[l]
 }
$$
are commutative in~$\sL$.

 Hence one can choose an index $\xi_1\in\Xi$ together with an arrow
$\xi_0\rarrow\xi_1$ in $\Xi$ such that the composition
$F(S_{\xi_0})\rarrow F(K)\rarrow G(K)$ factorizes through
the morphism $G(\xi_1)\rarrow G(K)$:
$$
 \xymatrix{
  F(S_{\xi_0}) \ar[r] \ar[d]_\sigma & F(K) \ar[d]^\phi \\
  G(S_{\xi_1}) \ar[r] & G(K) 
 }
$$
and the two pentagonal diagrams
$$
 \xymatrix{
  F(R') \ar[rr] \ar[d]_{\rho'} && F(S_{\xi_0}) \ar[d]^\sigma &&
  F(R'') \ar[ll] \ar[d]^{\rho''} \\
  G(R') \ar[r] & G(S_{\xi_0}) \ar[r] & G(S_{\xi_1})
  & G(S_{\xi_0}) \ar[l] & G(R'') \ar[l]
 }
$$
are commutative in~$\sL$.
 Then it remains to put $S=S_{\xi_0}$ and $T=S_{\xi_1}$, and refer
to Lemma~\ref{inserter-main-lemma}.
\end{proof}

 Finally, we are ready to prove the theorem.

\begin{proof}[Proof of Theorem~\ref{inserter-theorem}]
 By Proposition~\ref{inserter-obvious-inclusion}, all the pairs
$(S,\psi)\in\sE$ with $S\in\sK_{<\kappa}$ are $\kappa$\+presentable
in~$\sE$.
 It is also clear that the full subcategory $\sE'_{<\kappa}$ of
all such pairs $(S,\psi)$ is closed under retracts in~$\sE$ (since
the full subcategory $\sK_{<\kappa}$ is closed under retracts
in~$\sK$).
 Let $\sS\subset\sE$ be a set of representatives of isomorphism
classes of objects from $\sE'_{<\kappa}$.
 In view of~\cite[Remarks~1.9 and~2.2(4)]{AR} (see the discussion
in Section~\ref{preliminaries-secn}), it suffices to prove that,
for every object $E\in\sE$, the indexing category $\Delta=\Delta_E$
of the canonical diagram $C=D_E$ of morphisms into $E$ from objects
of $\sS$ is $\kappa$\+filtered, and that $E=\varinjlim_{v\in\Delta}C_v$.

 The former assertion is the result of
Proposition~\ref{inserter-diagram-kappa-filtered}.
 To prove the latter one, notice that by
Lemma~\ref{accessible-canonical-colimit}
we have $K=\varinjlim_{w\in\Delta_K}D_w$ in $\sK$, where
$D\:\Delta_K\rarrow\sK$ is the canonical diagram of morphisms
into $K$ from representatives of isomorphisms classes of objects
from~$K_{<\kappa}$.
 Since the natural functor $\delta\:\Delta\rarrow\Delta_K$ between
the indexing categories is cofinal by
Proposition~\ref{inserter-diagram-cofinal}, it follows that
$K=\varinjlim_{v\in\Delta_E}D_{\delta(v)}$ in~$\sK$.
 As the forgetful functor $\sE\rarrow\sK$ is conservative and
preserves $\kappa$\+filtered colimits, we can conclude that
$E=\varinjlim_{v\in\Delta_E}C_v$ in~$\sE$.
\end{proof}

\begin{rem} \label{joint-inserter}
 In applications of Theorem~\ref{inserter-theorem}, one may be
interested in the \emph{joint inserter} of a family of pairs
of functors.
 Let $\sK$ be a $\kappa$\+accessible category and $(\sL_i)_{i\in I}$
be a family of $\kappa$\+accessible categories.
 Let $F_i$, $G_i\:\sK\rightrightarrows\sL_i$ be a family of pairs of
parallel functors, all of them preserving $\kappa$\+directed colimits
and colimits of $\lambda$\+indexed chains.
 Assume further that the functors $F_i$ take $\kappa$\+presentable
objects to $\kappa$\+presentable objects, and that the cardinality of
the indexing set $I$ is smaller than~$\kappa$.

 Let $\sE$ be the category of pairs $(K,\phi)$, where $K\in\sK$ is
an object and $\phi=(\phi_i)_{i\in I}$ is a family of morphisms
$\phi_i\:F_i(K)\rarrow G_i(K)$ in~$\sL_i$.
 Then the category $\sE$ is $\kappa$\+accessible, and
the $\kappa$\+presentable objects of $\sE$ are precisely all
the pairs $(S,\psi)$ where $S$ is a $\kappa$\+presentable object
of~$\sK$.
 This assertion can be deduced from
Proposition~\ref{product-proposition} and Theorem~\ref{inserter-theorem}
by passing to the Cartesian product category $\sL=\prod_{i\in I}\sL_i$.
 The family of functors $F_i\:\sK\rarrow\sL_i$ defines a functor
$F\:\sK\rarrow\sL$, and the family of functors $G_i\:\sK\rarrow\sL_i$
defines a functor $G\:\sK\rarrow\sL$.
 It follows from Proposition~\ref{product-proposition} that all
the assumptions of Theorem~\ref{inserter-theorem} are satisfied by
the category $\sL$ and the pair of functors $F$,~$G$.
\end{rem}

\Section{Pseudopullback} \label{pseudopullback-secn}

 As in Sections~\ref{equifier-secn} and~\ref{inserter-secn}, we
consider a regular cardinal~$\kappa$ and a smaller infinite
cardinal $\lambda<\kappa$.
 Let $\sA$, \,$\sB$, and $\sC$ be $\kappa$\+accessible categories
in which all $\lambda$\+indexed chains (of objects and morphisms)
have colimits.
 Let $\Theta_\sA:\sA\rarrow\sC$ and $\Theta_\sB\:\sB\rarrow\sC$ be two
functors preserving $\kappa$\+directed colimits and colimits of
$\lambda$\+indexed chains, and taking $\kappa$\+presentable objects
to $\kappa$\+presentable objects.

 Let $\sD$ be the category of triples $(A,B,\theta)$, where
$A\in\sA$ and $B\in\sB$ are objects and
$\theta\:\Theta_\sA(A)\simeq\Theta_\sB(B)$ is an isomorphism in~$\sC$.
 This construction of the category $\sD$ is known as
the \emph{pseudopullback}~\cite[Proposition~3.1]{CR},
\cite[Section~2]{RR}.
 The aim of this section is to deduce the following corollary
of Theorems~\ref{equifier-theorem} and~\ref{inserter-theorem}.

\begin{cor} \label{pseudopullback-corollary}
 In the assumptions above, the category\/ $\sD$ is\/
$\kappa$\+accessible.
 The $\kappa$\+presentable objects of\/ $\sD$ are precisely
all the triples $(A,B,\theta)$, where $A$ is a $\kappa$\+presentable
object of\/ $\sA$ and $B$ is a $\kappa$\+presentable object of\/~$\sB$.
\end{cor}

\begin{proof}
 This result, going back to~\cite[Remark~3.2(I), Theorem~3.8,
Corollary~3.9, and Remark~3.11(II)]{Ulm}, appears in the recent
literature as~\cite[Proposition~3.1]{CR},
\cite[Pseudopullback Theorem~2.2]{RR}.
 So we include this proof for the sake of completeness of
the exposition and for illustrative purposes.

 The point is that the pseudopullback can be constructed as
a combination of products, inserters, and equifiers.
 Put $\sK=\sA\times\sB$ and $\sL=\sC\times\sC$, and consider
the following pair of parallel functors $F$, $G\:\sK\rarrow\sL$.
 The functor $F$ takes a pair of objects $(A,B)\in\sA\times\sB$
to the pair of objects $(\Theta_\sA(A),\Theta_\sB(B))\in\sC\times\sC$.
 The functor $G$ takes a pair of objects $(A,B)\in\sA\times\sB$ to
the pair of objects $(\Theta_\sB(B),\Theta_\sA(A))\in\sC\times\sC$.
 Then the related inserter category $\sE$ from
Section~\ref{inserter-secn} (cf.\ Remark~\ref{joint-inserter}) is
the category of quadruples $(A,B,\theta',\theta'')$, where $A\in\sA$
and $B\in\sB$ are objects, while $\theta'\:\Theta_\sA(A)\rarrow
\Theta_\sB(B)$ and $\theta''\:\Theta_\sB(B)\rarrow\Theta_\sA(A)$ are
arbitrary morphisms.

 Theorem~\ref{inserter-theorem} together with
Proposition~\ref{product-proposition} tell that the category $\sE$ is
$\kappa$\+presentable, and the $\kappa$\+presentable objects of $\sE$
are precisely all the quadruples $(A,B,\theta',\theta'')$ such that
$A$ is a $\kappa$\+presentable object of $\sA$ and $B$ is
a $\kappa$\+presentable object of~$\sB$.

 It remains to apply the joint equifier construction of
Section~\ref{equifier-secn} and Remark~\ref{joint-equifier} to
the family of two pairs of parallel natural transformations
$(\id,\>\theta'\circ\theta'')$ and $(\id,\>\theta''\circ\theta')$ of
functors $\sE\rarrow\sC$ in order to produce the full subcategory
$\sD\subset\sE$ of all quadruples $(A,B,\theta',\theta'')$ such that
$\theta'$ and~$\theta''$ are mutually inverse isomorphisms
$\Theta_\sA(A)\simeq\Theta_\sB(B)$.
 Then Theorem~\ref{equifier-theorem} tells that the category $\sD$
is $\kappa$\+accessible and describes its full subcategory of
$\kappa$\+presentable objects, as desired.
\end{proof}

\begin{rem} \label{isomorpher-remark}
 Alternatively, one can consider what we would call
the \emph{isomorpher} construction for two parallel functors between
two categories $P$, $Q\:\sH\rightrightarrows\sG$.
 (It appears in the literature under the name of
the ``iso-inserter''~\cite[Section~4]{Kel}, \cite[Section~1]{BKPS}.)
 The isomorpher category $\sD$ consists of all pairs $(H,\theta)$,
where $H\in\sH$ is an object and $\theta\:P(H)\simeq Q(H)$ is
an isomorphism in~$\sG$.

 One can observe that the pseudopullback and the isomorpher
constructions are actually equivalent,  in the sense that they can be
reduced to one another.
 Given a pair of functors $\Theta_\sA\:\sA\rarrow\sC$ and $\Theta_\sB\:
\sB\rarrow\sC$, one can put $\sH=\sA\times\sB$ and $\sG=\sC$, and
denote by $P\:\sH\rarrow\sG$ and $Q\:\sH\rarrow\sG$ the compositions
$\sA\times\sB\rarrow\sA\rarrow\sC$ and $\sA\times\sB\rarrow\sB
\rarrow\sC$.
 In this context, the two constructions of the category $\sD$ agree.

 Conversely, given a pair of parallel functors $P$, $Q\:\sH
\rightrightarrows\sG$, put $\sA=\sB=\sH$ and $\sC=\sH\times\sG$.
 Let the functor $\Theta_\sA\:\sA\rarrow\sC$ take an object $H'\in\sH$
to the pair $(H',P(H'))\in\sH\times\sG$ and the functor $\Theta_\sB\:
\sB\rarrow\sC$ take an object $H''\in\sH$ to the pair
$(H'',Q(H''))\in\sH\times\sG$.
 Then an isomorphism $\Theta_\sA(H')\simeq\Theta_\sB(H'')$ in $\sC$
means a pair of isomorphisms $H'\simeq H''$ in $\sH$ and
$P(H')\simeq Q(H'')$ in~$\sG$.
 Up to a category equivalence, the datum of two objects $H'$, $H''
\in\sH$ endowed with such two isomorphisms is the same thing as
a single object $H\in\sH$ together with an isomorphism $P(H)\simeq Q(H)$
in~$\sG$.
 Thus, in this context, the two constructions of the category $\sD$
agree as well.

 Assume that the categories $\sH$ and $\sG$ are $\kappa$\+accessible
with colimits of $\lambda$\+indexed chains (for a regular
cardinal~$\kappa$ and a smaller infinite cardinal $\lambda<\kappa$).
 Assume further that the functors $F$ and $G$ preserve
$\kappa$\+directed colimits and colimits of $\lambda$\+indexed chains,
and that they take $\kappa$\+presentable objects to
$\kappa$\+presentable objects.
 Then it follows from Proposition~\ref{product-proposition}
and Corollary~\ref{pseudopullback-corollary} that the isomorpher
category $\sD$ is $\kappa$\+accessible, and the $\kappa$\+presentable
objects of $\sD$ are precisely all the pairs $(H,\theta)$ with
$H\in\sH_{<\kappa}$.
\end{rem}

\Section{Diagram Categories} \label{diagram-secn}

 In this section, we discuss two constructions: the category of
functors $\Fun(\sC,\sK)$ and the category of $k$\+linear functors
$\Fun_k(\sA,\sK)$.
 The former one is of interest to the general category theory, while
the latter one is relevant for additive category theory, module
theory, complexes in additive categories, etc.

 Let us start with the nonadditive case.
 Given a small category $\sC$ and a category $\sK$, we denote by
$\Fun(\sC,\sK)$ the category of functors $\sC\rarrow\sK$.

 Recall that a category $\sK$ is called \emph{locally
$\kappa$\+presentable}~\cite[Definitions~1.9 and~1.17]{AR}
if $\sK$ is $\kappa$\+accessible and all colimits exist in~$\sK$.
 The following theorem is a generalization of~\cite[Theorem~1.2]{Hen}
from the case of locally $\kappa$\+presentable categories to the case
of $\kappa$\+accessible categories with colimits of $\lambda$\+indexed
chains (for some fixed infinite cardinal $\lambda<\kappa$).
 It is also a correct version of~\cite[Lemma~5.1]{Mak} (which was
shown to be erroneous in full generality in~\cite[Theorem~1.3]{Hen}).

 A category $\sC$ is said to be \emph{$\kappa$\+small} if
the cardinality of the set of all objects and morphisms in $\sC$
is smaller than~$\kappa$.

\begin{thm} \label{nonadditive-diagram-theorem}
 Let $\kappa$~be a regular cardinal and\/ $\lambda<\kappa$ be a smaller
infinite cardinal.
 Let\/ $\sC$ be a $\kappa$\+small category.
 Let\/ $\sK$ be a $\kappa$\+accessible category in which all
$\lambda$\+indexed chains (of objects and morphisms) have colimits.
 Then the category\/ $\Fun(\sC,\sK)$ is $\kappa$\+accessible.
 The full subcategory\/ $\Fun(\sC,\sK_{<\kappa})$ is precisely
the full subcategory of all $\kappa$\+presentable objects in\/
$\Fun(\sC,\sK)$.
\end{thm}

\begin{proof}
 Similarly to the proof Corollary~\ref{pseudopullback-corollary},
the point is that the diagram category can be constructed as
a combination of products, inserters, and equifiers.
 Let $\sK'=\prod_{c\in\sC}\sK$ be the Cartesian product of copies of
the category $\sK$ indexed by the objects of the category~$\sC$,
and let $\sL'=\prod_{(c\to d)\in\sC}\sK$ be the similar product of
copies of $\sK$ indexed by the morphisms of the category~$\sC$.
 Proposition~\ref{product-proposition} tells that the categories
$\sK'$ and $\sL'$ are $\kappa$\+accessible, and describes their full
subcategories of $\kappa$\+presentable objects.

 Define a pair of parallel functors $F$, $G\:\sK'\rarrow\sL'$ as
follows.
 The functor $F$ assigns to a collection of objects
$(K_c\in\sK)_{c\in\sC}\in\sK'$ the collection of objects
$(L_{c\to d})_{(c\to d)\in\sC}\in\sL'$ given by the rules
$L_{c\to d}=K_c$ for any morphism $c\rarrow d$ in~$\sC$.
 Similarly, the functor $G$ assigns to a collection of objects
$(K_c\in\sK)_{c\in\sC}\in\sK'$ the collection of objects
$(L_{c\to d})_{(c\to d)\in\sC}\in\sL'$ given by the rules
$L_{c\to d}=K_d$ for any morphism $c\rarrow d$ in~$\sC$.

 Then the related inserter category $\sE$ from
Section~\ref{inserter-secn} (cf.\ Remark~\ref{joint-inserter})
is the category of all ``nonmultiplicative functors''
$\sC\rarrow\sK$.
 An object $E\in\sE$ is a rule assigning to every object $c\in\sC$
an object $E_c\in\sK$ and to every morphism $c\rarrow d$ in $\sC$
a morphism $E_c\rarrow E_d$ in~$\sK$.
 The conditions of compatibility with the compositions of morphisms
and with the identity morphisms are \emph{not} imposed.
 Morphisms of ``nonmultiplicative functors'' (i.~e., the morphisms
in~$\sE$) are similar to the usual morphisms of functors; so
the desired functor category $\Fun(\sC,\sK)$ is a full subcategory
in~$\sE$.

 Theorem~\ref{inserter-theorem} tells that the category $\sE$ is
$\kappa$\+accessible, and describes its full subcategory of
$\kappa$\+presentable objects.
 Now the desired full subcategory $\Fun(\sC,\sK)\subset\sE$ can be
produced as a joint equifier category, as in
Section~\ref{equifier-secn} and Remark~\ref{joint-equifier}.
 There are two kinds of pairs of parallel natural transformations
to be equified.

 Firstly, for every composable pair of morphisms $b\rarrow c\rarrow d$
in $\sC$, we have a pair of parallel functors $F_{b\to c\to d}$,
$G_{b\to c\to d}\:\sE\rightrightarrows\sK$ and a pair of parallel
natural transformations $\phi_{b\to c\to d}$, $\psi_{b\to c\to d}\:
F_{b\to c\to d}\rightrightarrows G_{b\to c\to d}$.
 The functor $F_{b\to c\to d}$ takes an object $E\in\sE$ to
the object $E_b\in\sK$, and the functor $G_{b\to c\to d}$ takes
an object $E\in\sE$ to the object $E_d\in\sK$.
 The natural transformation $\phi_{b\to c\to d}$ acts by the composition
of the morphisms $E_b\rarrow E_c\rarrow E_d$ in $\sK$ assigned to
the morphisms $b\rarrow c$ and $c\rarrow d$ by the datum of
the object~$E$.
 The natural transformation $\psi_{b\to c\to d}$ acts by the morphism
$E_b\rarrow E_d$ assigned to the composition of the morphisms $b\rarrow
c\rarrow d$ in $\sC$ by the datum of the object~$E$.

 Secondly, for every object $c\in\sC$, we have a pair of parallel
functors $F_c=G_c\:\sE\rarrow\sK$ and a pair of parallel natural
transformations $\phi_c$, $\psi_c\:F_c\rightrightarrows G_c$.
 The functor $F_c=G_c$ takes an object $E\in\sE$ to
the object $E_c\in\sK$.
 The natural transformation~$\phi_c$ acts by the morphism
$E_c\rarrow E_c$ in $\sK$ assigned to the identity morphism~$\id_c$
in $\sC$ by the datum of the object~$E$; while~$\psi_c$ is
the identity natural transformation.

 The resulting joint equifier is the functor category $\Fun(\sC,\sK)$.
 Theorem~\ref{equifier-theorem} tells that this category is
$\kappa$\+accessible, and provides the desired description of its
full subcategory of $\kappa$\+presentable objects.
\end{proof}

 Now let $k$~be a commutative ring.
 A \emph{$k$\+linear category} $\sA$ is a category enriched in
$k$\+modules.
 This means that, for any two objects $a$ and $b\in\sA$, the set of
morphisms $\Hom_\sA(a,b)$ is a $k$\+module, and the composition
maps $\Hom_\sA(b,c)\times\Hom_\sA(a,b)\rarrow\Hom_\sA(a,c)$ are
$k$\+bilinear.

 Suppose given a set of objects~$a$ and, for every pair of objects
$a$, $b$, a \emph{generating set} of morphisms $\Gen(a,b)$.
 Then one can construct the $k$\+linear category $\sB$ on the given
set of objects \emph{freely generated} by the given generating sets
of morphisms.
 For every pair of objects $a$, $b$, the free $k$\+module
$\Hom_\sB(a,b)$ has a basis consisting of all the formal compositions
$g_n\dotsm g_1$, \,$n\ge0$, where $g_i\in\Gen(c_i,c_{i+1})$,
\ $c_1=a$, \,$c_{n+1}=b$.

 Furthermore, suppose given a \emph{set of defining relations}
$\Rel(a,b)\subset\Hom_\sB(a,b)$ for every pair of objects $a$,~$b$.
 Then one can construct the two-sided ideal of morphisms $\sJ\subset
\sB$ generated by all the relations, and pass to the $k$\+linear
quotient category $\sA=\sB/\sJ$ by the ideal~$\sJ$.

 Abusing terminology, we will say that a $k$\+linear category $\sA$
is \emph{$\kappa$\+presented} if it has the form $\sA=\sB/\sJ$ as
per the construction above, where the set of objects~$\{a\}$, the set
of all generators $\coprod_{a,b}\Gen(a,b)$, and the set of all
relations $\coprod_{a,b}\Rel(a,b)$ all have the cardinalities
smaller than~$\kappa$.
 In another terminology, one could say that $\sA$ is ``the path
category of a $\kappa$\+small quiver with a $\kappa$\+small set of
relations''.

 A $k$\+linear category $\sK$ is said to be \emph{$\kappa$\+accessible}
if it is $\kappa$\+accessible as an abstract category.
 Given a small $k$\+linear category $\sA$ and a $k$\+linear category
$\sK$, we denote by $\Fun_k(\sA,\sK)$ the ($k$\+linear) category of
$k$\+linear functors $\sA\rarrow\sK$.
 The following theorem is a $k$\+linear version of
Theorem~\ref{nonadditive-diagram-theorem}.

\begin{thm} \label{k-linear-diagram-theorem}
 Let $\kappa$~be a regular cardinal and $\lambda<\kappa$ be a smaller
infinite cardinal.
 Let $k$~be a commutative ring, let\/ $\sA$ be a $\kappa$\+presented
$k$\+linear category, and let\/ $\sK$ be a $\kappa$\+accessible
$k$\+linear category in which all $\lambda$\+indexed chains have
colimits.
 Then the category\/ $\Fun_k(\sA,\sK)$ is $\kappa$\+accessible.
 The full subcategory\/ $\Fun_k(\sA,\sK_{<\kappa})$ is precisely
the full subcategory of all $\kappa$\+presentable objects in\/
$\Fun_k(\sA,\sK)$.
\end{thm}

\begin{proof}
 The argument is similar to the proof of
Theorem~\ref{nonadditive-diagram-theorem}, with the only difference
that one works with the generating morphisms and defining relations
in $\sA$ instead of all morphisms and all compositions in~$\sC$.
 Let $\sK'=\prod_{a\in\sA}\sK$ be the Cartesian product of copies
of the category $\sK$ indexed by the objects of the category $\sA$,
and let $\sL'=\prod_{a,b\in\sA}\prod_{(a\to b)\in\Gen(a,b)}\sK$
be the similar product of copies of $\sK$ indexed by the set of
generating morphisms $\coprod_{a,b}\Gen(a,b)$.
 Proposition~\ref{product-proposition} tells that the categories
$\sK'$ and $\sL'$ are $\kappa$\+accessible, and describes their full
subcategories of $\kappa$\+presentable objects.
 
 Define a pair of parallel functors $F$, $G\:\sK'\rarrow\sL'$ as
follows.
 The functor $F$ assigns to a collection of objects
$(K_a\in\sK)_{a\in\sA}\in\sK'$ the collection of objects
$(L_{a\to b})_{(a\to b)\in\Gen(a,b),\,a,b\in\sA}\in\sL'$ given by
the rules $L_{a\to b}=K_a$ for any generating morphism
$(a\to b)\in\Gen(a,b)$.
 Similarly, the functor $G$ assigns to a collection of objects
$(K_a\in\sK)_{a\in\sA}\in\sK'$ the collection of objects
$(L_{a\to b})_{(a\to b)\in\Gen(a,b),\,a,b\in\sA}\in\sL'$ given by
the rules $L_{a\to b}=K_b$ for any generating morphism
$(a\to b)\in\Gen(a,b)$.

 Then the related inserter category $\sE$ from
Section~\ref{inserter-secn} (cf.\ Remark~\ref{joint-inserter}) is
naturally equivalent to the category $\Fun_k(\sB,\sK)$, where $\sB$
is the ``path category of the quiver without relations'' constructed
in the discussion preceding the formulation of the theorem.
 Theorem~\ref{inserter-theorem} tells that the category $\sE$ is
$\kappa$\+accessible, and defines its full subcategory of
$\kappa$\+presentable objects.
 The category $\Fun_k(\sA,\sK)$ we are interested in is a full
subcategory in $\sE=\Fun_k(\sB,\sK)$ consisting of all the ``quiver
representations in $\sK$ for which the relations are satisfied''.
 The full subcategory $\Fun_k(\sA,\sK)\subset\Fun_k(\sB,\sK)$ can be
produced as a joint equifier category, as in
Section~\ref{equifier-secn} and Remark~\ref{joint-equifier}.

 The pairs of parallel natural transformations to be equified are
indexed by elements of the set of defining relations
$\coprod_{a,b}\Rel(a,b)$.
 Given a defining relation $r\in\Rel(a,b)$, we have a pair of
parallel functors $F_r$, $G_r\:\sE\rightrightarrows\sK$ and a pair of
natural transformations $\phi_r$, $\psi_r\:F_r\rightrightarrows G_r$.
 The functor $F_r\:\Fun_k(\sB,\sK)\rarrow\sK$ takes a functor
$E\:\sB\rarrow\sK$ to the object $E(a)\in\sK$, and the functor
$G_r$ takes the functor $E$ to the object $E(b)\in\sK$.
 The natural transformation~$\phi_r$ acts by the morphism
$E(r)\:E(a)\rarrow E(b)$.
 The natural transformation~$\psi_r$ acts by the zero morphism
$0\:E(a)\rarrow E(b)$ in the $k$\+linear category~$\sK$.

 The resulting joint equifier is the category of $k$\+linear functors
$\Fun_k(\sA,\sK)$.
 Theorem~\ref{equifier-theorem} tells that this category is
$\kappa$\+accessible, and provides the desired description of its
full subcategory of $\kappa$\+presentable objects.
\end{proof}

\Section{Brief Preliminaries on $2$-Categories}
\label{2cat-prelim-secn}

 The aim of this section is to provide a very brief and mostly
terminological preliminary discussion for the purposes of the next
two Sections~\ref{conical-secn}\+-\ref{weighted-secn}.
 The reader can find the details by following the references.

 Throughout the three sections, for the most part we adopt the policy
of benign neglect with respect to set-theoretical issues of size
(i.~e., the distinction between sets and classes).
 When specific restrictions on the size matter, we mention them.

 In the terminology of higher category theory, the prefix ``$2$\+''
means strict concepts, while the prefix ``bi\+'' refers to relaxed ones.
 So $2$\+categories are strict, while bicategories are
relaxed~\cite{Ben}.

 A \emph{$2$\+category} is a category enriched in the category of
categories $\Cat$ (with the monoidal structure on $\Cat$ given by
the Cartesian product)~\cite{KS}.
 In particular, there is the important \emph{$2$\+category of 
categories}~$\TwoCat$: categories are the objects, functors are
the $1$\+cells, natural transformations are the $2$\+cells.

 In the terminology of the bicategory theory, one speaks of
\emph{morphisms of bicategories} (which are multiplicative and unital
on $1$\+cells up to coherent families of $2$\+cells) or 
\emph{homomorphisms of bicategories} (which are multiplicative and
unital on $1$\+cells up to coherent families of \emph{invertible}
$2$\+cells)~\cite[Section~4]{Ben}.
 Even when one is only interested in $2$\+categories, the notion of
a $2$\+functor may be too strict, and one may want to relax it by
considering morphisms of $2$\+categories (known as \emph{lax functors}),
or homomorphisms of $2$\+categories (known as \emph{pseudofunctors}).

 Let $\Gamma$ and $\Delta$ be two $2$\+categories.
 Then $2$\+functors $\Gamma\rarrow\Delta$ form a $2$\+category
$[\Gamma,\Delta]$.
 The objects of $[\Gamma,\Delta]$ are the $2$\+functors $\Gamma\rarrow
\Delta$, the $1$\+cells of $[\Gamma,\Delta]$ are the $2$\+natural
transformations, and the $2$\+cells of $[\Gamma,\Delta]$ are called
\emph{modifications}~\cite[Section~1.4]{KS}.
 A $2$\+functor $\Gamma\rarrow\Delta$ is a rule assigning to every
object of $\Gamma$ an object of $\Delta$, to every $1$\+cell of
$\Gamma$ a $1$\+cell of $\Delta$, and to every $2$\+cell of $\Gamma$
a $2$\+cell of~$\Delta$.
 A $2$\+natural transformation is a rule assiging to every object
of $\Gamma$ a $1$\+cell in~$\Delta$.
 A modification is a rule assigning to every object of $\Gamma$
a $2$\+cell in~$\Delta$.
 $2$\+categories and $2$\+functors form the \emph{$3$\+category of\/
$2$\+categories}: $2$\+categories are the objects, $2$\+functors are
the $1$\+cells, $2$\+natural transformations are the $2$\+cells,
and modifications are the $3$\+cells.

 Even when one is only interested in $2$\+functors rather than
the more relaxed concepts of lax functors or pseudofunctors, the notion
of a $2$\+natural transformation may be too strict, and one may want
to relax it.
 Then one can consider \emph{lax natural transformations} (compatible
with the action of the $2$\+functors on $1$\+cells in $\Gamma$ up
to a coherent family of $2$\+cells in~$\Delta$) or \emph{pseudonatural
transformations} (compatible with the action of the $2$\+functors on
the $1$\+cells in $\Gamma$ up to a coherent family of \emph{invertible}
$2$\+cells in~$\Delta$).
 In the terminology of~\cite[\S4.1]{MP}, lax natural
transformations are called ``transformations'', pseudonatural 
transformations are called ``strong transformations'', and
$2$\+natural transformations are called ``strict transformations''.
 The $2$\+category of $2$\+functors $\Gamma\rarrow\Delta$,
pseudonatural transformations, and modifications is denoted by
$\Psd[\Gamma,\Delta]$ in~\cite{Bir}, \cite[Section~5]{Kel},
\cite[Section~2]{BKPS}.

 In connection with the ``lax'' notions, the choice of the direction
of the (possibly noninvertible) $2$\+cells providing the relaxed
compatibility becomes important.
 When the direction is reversed, the correspoding notions are called
``oplax''.
 For ``pseudo'' notions, the compatibility $2$\+cells are assumed to be
invertible, and so the choice of the direction in which they act
no longer matters.

\Section{Conical Pseudolimits, Lax Limits, and Oplax Limits}
\label{conical-secn}

 We denote by $\TwoCat$ the $2$\+category of small categories and
by $\TwoCAT$ the $2$\+category of locally small categories (i.~e.,
large categories in which morphisms between any fixed pair of objects
form a set).
 So the categories of morphisms in $\TwoCAT$ need not be even locally
small; this will present no problem for our constructions.

 Let $\Gamma$ be a small $2$\+category and
$H\:\Gamma\rarrow\TwoCAT$ be a $2$\+functor.
 The (\emph{conical}) \emph{lax limit} of $H$ is a category $\sL$
whose objects are the following sets of data:
\begin{enumerate}
\renewcommand{\theenumi}{\roman{enumi}}
\item for every object $\gamma\in\Gamma$, an object $L_\gamma\in
H(\gamma)$ of the category $H(\gamma)$ is given;
\item for every $1$\+cell $a\:\gamma\rarrow\delta$ in $\Gamma$,
a morphism $l_a\:H(a)(L_\gamma)\rarrow L_\delta$ in the category
$H(\delta)$ is given.
\end{enumerate}
 Here $H(a)\:H(\gamma)\rarrow H(\delta)$ is the functor assigned to
the $1$\+cell $a\:\gamma\rarrow\delta$ by the $2$\+functor~$H$.

 The set of data~(i\+-ii) must satisfy the following conditions:
\begin{enumerate}
\renewcommand{\theenumi}{\roman{enumi}}
\setcounter{enumi}{2}
\item for every identity $1$\+cell $a=\id_\gamma\:\gamma\rarrow\gamma$
in $\Gamma$, one has $l_{\id_\gamma}=\id_{L_\gamma}\:L_\gamma\rarrow
L_\gamma$;
\item for every composable pair of $1$\+cells $a\:\gamma\rarrow\delta$
and $b\:\delta\rarrow\epsilon$ in $\Gamma$, one has
$l_{ba}=l_b\circ H(b)(l_a)$ in the category $H(\epsilon)$;
\item for every $2$\+cell $t\:a\rarrow b$, where $a$, $b\:\gamma
\rightrightarrows\delta$ is a pair of parallel $1$\+cells in $\Gamma$, 
the triangular diagram
$$
 \xymatrix{
  H(a)(L_\gamma) \ar[dd]_{H(t)_{L_\gamma}} \ar[rd]^-{l_a} \\
  & L_\delta \\
  H(b)(L_\gamma) \ar[ru]_-{l_b}
 }
$$
is commutative in the category~$H(\delta)$.
\end{enumerate}
 Here $H(t)\:H(a)\rarrow H(b)$ is the morphism of functors from
the category $H(\gamma)$ to the category $H(\delta)$ assigned
to the $2$\+cell $t\:a\rarrow b$ by the $2$\+functor~$H$.

 A morphism $L\rarrow M$ in the category $\sL$ is the datum of
a morphism $L_\gamma\rarrow M_\gamma$ in the category $H(\gamma)$ for
every object $\gamma\in\Gamma$, satisfying the obvious compatibility
condition with the data~(ii) for the objects $L$ and~$M$.

 The (\emph{conical}) \emph{oplax limit} of the $2$\+functor $H$ is
the category $\sM$ whose objects are the following sets of data:
\begin{enumerate}
\renewcommand{\theenumi}{\roman{enumi}${}^*$}
\item for every object $\gamma\in\Gamma$, an object $M_\gamma\in
H(\gamma)$ of the category $H(\gamma)$ is given;
\item for every $1$\+cell $a\:\gamma\rarrow\delta$ in $\Gamma$,
a morphism $m_a\:M_\delta\rarrow H(a)(M_\gamma)$ in the category
$H(\delta)$ is given.
\end{enumerate}

 The set of data~(i${}^*$\+-ii${}^*$) must satisfy the following
 conditions:
\begin{enumerate}
\renewcommand{\theenumi}{\roman{enumi}${}^*$}
\setcounter{enumi}{2}
\item for every identity $1$\+cell $a=\id_\gamma\:\gamma\rarrow\gamma$
in $\Gamma$, one has $m_{\id_\gamma}=\id_{M_\gamma}\:
M_\gamma\rarrow M_\gamma$;
\item for every composable pair of $1$\+cells $a\:\gamma\rarrow\delta$
and $b\:\delta\rarrow\epsilon$ in $\Gamma$, one has
$m_{ba}=H(b)(m_a)\circ m_b$ in the category $H(\epsilon)$;
\item for every $2$\+cell $t\:a\rarrow b$, where $a$, $b\:\gamma
\rightrightarrows\delta$ is a pair of parallel $1$\+cells in $\Gamma$, 
the triangular diagram
$$
 \xymatrix{
  & H(a)(M_\gamma) \ar[dd]^{H(t)_{M_\gamma}}  \\
  M_\delta \ar[ru]^-{m_a} \ar[rd]_-{m_b} \\
  & H(b)(M_\gamma) 
 }
$$
is commutative in the category~$H(\delta)$.
\end{enumerate}

 A morphism $L\rarrow M$ in the category $\sM$ is the datum of
a morphism $L_\gamma\rarrow M_\gamma$ in the category $H(\gamma)$ for
every object $\gamma\in\Gamma$, satisfying the obvious compatibility
condition with the data~(ii${}^*$) for the objects $L$ and~$M$.

 The \emph{pseudolimit} of the $2$\+functor $H$ is the full subcategory
$\sE\subset\sL$ consisting of all the objects $E\in\sL$ such that
the morphism $e_a\:H(a)(E_\gamma)\rarrow E_\delta$ in~(ii) is
an isomorphism in $H(\delta)$ for every $1$\+cell
$a\:\gamma\rarrow\delta$ in~$\Gamma$.
 Equivalently, the pseudolimit $\sE$ can be defined as the full
subcategory $\sE\subset\sM$ consisting of all the objects $E\in\sM$
such that the morphism $e_a\:E_\delta\rarrow H(a)(E_\gamma)$
in~(ii${}^*$) is an isomorphism in $H(\delta)$ for every $1$\+cell
$a\:\gamma\rarrow\delta$ in~$\Gamma$.

 Let $\kappa$~be a regular cardinal and $\lambda<\kappa$ be a smaller
infinite cardinal.
 Denote by $\ACC_{\lambda,\kappa}\subset\TwoCAT$ the following
$2$\+subcategory in $\TwoCAT$.
 The objects of $\ACC_{\lambda,\kappa}$ are all
the $\kappa$\+accessible categories with colimits of
$\lambda$\+indexed chains.
 The $1$\+cells of $\ACC_{\lambda,\kappa}$ are the functors
preserving $\kappa$\+directed colimits and colimits of
$\lambda$\+indexed chains.
 The $2$\+cells of $\ACC_{\lambda,\kappa}$ are the (arbitrary)
natural transformations.

 As usual, we will say that a $2$\+category is \emph{$\kappa$\+small}
if it has less than $\kappa$~objects, less than~$\kappa$\,
$1$\+cells, and less than~$\kappa$\, $2$\+cells.

\begin{thm} \label{conical-oplax-limit-theorem}
 Let $\kappa$~be a regular cardinal and $\lambda<\kappa$ be a smaller
infinite cardinal.
 Let $\Gamma$ be a $\kappa$\+small\/ $2$\+category and
$H\:\Gamma\rarrow\ACC_{\lambda,\kappa}$ be a\/ $2$\+functor.
 Then the oplax limit\/ $\sM$ of the\/ $2$\+functor $H$ (computed in\/
$\TwoCAT$, as per the construction above) belongs to\/
$\ACC_{\lambda,\kappa}$.
 For every object $\gamma\in\Gamma$, the natural forgetful/projection
functor $\sM\rarrow H(\gamma)$ belongs to\/ $\ACC_{\lambda,\kappa}$.
 An object $S\in\sM$ is $\kappa$\+presentable if and only if,
for every object $\gamma\in\Gamma$, the image $S_\gamma$ of $S$ in
$H(\gamma)$ is $\kappa$\+presentable.
\end{thm}

\begin{proof}
 Similarly to the proofs of Corollary~\ref{pseudopullback-corollary}
and Theorems~\ref{nonadditive-diagram-theorem}\+-%
\ref{k-linear-diagram-theorem}, one constructs the oplax limit $\sM$
as a combination of products, inserters, and equifiers.

 Let $\sK=\prod_{\gamma\in\Gamma}H(\gamma)$ be the Cartesian product
of the categories $H(\gamma)$ taken over all objects $\gamma\in\Gamma$,
and let $\sL=\prod_{(a:\gamma\to\delta)\in\Gamma}H(\delta)$ be
the Cartesian product of the categories $H(\delta)$ taken over all
the $1$\+cells $a\:\gamma\rarrow\delta$ in~$\Gamma$.
 Consider the following pair of parallel functors $F$, $G\:\sK\rarrow
\sL$.
 The functor $F$ takes a collection of objects
$(M_\gamma\in H(\gamma))_{\gamma\in\Gamma}\in\sK$ to the collection of
objects $(M_\delta\in H(\delta))_{(a:\gamma\to\delta)}\in\sL$.
 The functor $G$ takes the same collection of objects
$(M_\gamma\in H(\gamma))_{\gamma\in\Gamma}\in\sK$ to the collection of
objects $(F(a)(M_\gamma)\in H(\delta))_{(a:\gamma\to\delta)}\in\sL$.

 Then the related inserter category $\sE$ from
Section~\ref{inserter-secn} (cf.\ Remark~\ref{joint-inserter})
is the category of all sets of data~(i${}^*$\+-ii${}^*$) from
the definition of the oplax limit above.
 The conditions~(iii${}^*$\+-v${}^*$) have not been imposed yet.

 Theorem~\ref{inserter-theorem} tells that $\sE$ is
a $\kappa$\+accessible category and describes its full subcategory of
$\kappa$\+presentable objects.
 The desired oplax limit $\sM$ is a full subcategory $\sM\subset\sE$
which can be produced as a joint equifier category, as in
Section~\ref{equifier-secn} and Remark~\ref{joint-equifier}.
 There are three kinds of pairs of parallel natural transformations
to be equified, corresponding to the three
conditions~(iii${}^*$\+-v${}^*$).

 Firstly, for every object $\gamma\in\Gamma$, we have a pair of
parallel functors $F_\gamma=G_\gamma\:\sE\rarrow H(\gamma)$ and
a pair of parallel natural transformations $\phi_\gamma$,
$\psi_\gamma\:F_\gamma\rarrow G_\gamma$.
 The functor $F_\gamma=G_\gamma$ takes an object $E\in\sE$ to
the object $E_\gamma\in H(\gamma)$.
 The natural transformation $\phi_\gamma$ acts by the morphism
$e_{\id_\gamma}\:E_\gamma\rarrow E_\gamma$ assigned to the identity
$1$\+cell $\id_\gamma\:\gamma\rarrow\gamma$ in $\Gamma$ by
the datum~(ii$^*$) for the object $E\in\sE$; while $\psi_\gamma$
is the identity natural transformation.

 Secondly, for every composable pair of $1$\+cells $a\:\gamma\rarrow
\delta$ and $b\:\delta\rarrow\epsilon$ in $\Gamma$, we have a pair of
parallel functors $F_{a,b}$, $G_{a,b}\:\sE\rightrightarrows H(\epsilon)$
and a pair of parallel natural transformations $\phi_{a,b}$,
$\psi_{a,b}\:F_{a,b}\rightrightarrows G_{a,b}$.
 The functor $F_{a,b}$ takes an object $E\in\sE$ to the object
$E_\epsilon\in H(\epsilon)$.
 The functor $G_{a,b}$ takes an object $E\in\sE$ to the object
$H(ba)(E_\gamma)\in H(\epsilon)$.
 The natural transformation~$\phi_{a,b}$ acts by the morphism
$e_{ba}\:E_\epsilon\rarrow H(ba)(E_\gamma)$.
 The natural transformation~$\psi_{a,b}$ acts by the composition of
morphisms $H(b)(e_a)\circ e_b\:E_\epsilon\rarrow H(b)(E_\delta)
\rarrow H(ba)(E_\gamma)$.

 Thirdly, for every $2$\+cell $t\:a\rarrow b$, where $a$, $b\:\gamma
\rightrightarrows\delta$ is a pair of parallel $1$\+cells in $\Gamma$,
we have a pair of parallel functors $F_t$, $G_t\:\sE\rightrightarrows
H(\delta)$ and a pair of parallel natural transformations $\phi_t$,
$\psi_t\:F_t\rightrightarrows G_t$.
 The functor $F_t$ takes an object $E\in\sE$ to the object
$E_\delta\in H(\delta)$.
 The functor $G_t$ takes an object $E\in\sE$ to the object
$H(b)(E_\gamma)\in H(\delta)$.
 The natural transformation~$\phi_t$ acts by the composition of
morphisms $H(t)_{E_\gamma}\circ e_a\:E_\delta\rarrow H(a)(E_\gamma)
\rarrow H(b)(E_\gamma)$.
 The natural transformation~$\psi_t$ acts by the morphism
$e_b\:E_\delta\rarrow H(b)(E_\gamma)$.

 The resulting joint equifier is the oplax limit~$\sM$.
 Theorem~\ref{equifier-theorem} tells that this category is
$\kappa$\+accessible, and provides the desired description of its
full subcategory of $\kappa$\+presentable objects.
 This proves the first and the third assertions of the theorem,
while the second assertion is easy.
\end{proof}

 Denote by $\ACC_{\lambda,\kappa}^\kappa\subset\ACC_{\lambda,\kappa}$
the following $2$\+subcategory in $\TwoCAT$.
 The objects of $\ACC_{\lambda,\kappa}^\kappa$ are the same as
the objects of $\ACC_{\lambda,\kappa}$, i.~e., all
the $\kappa$\+accessible categories with colimits of
$\lambda$\+indexed chains.
 The $1$\+cells of $\ACC_{\lambda,\kappa}^\kappa$ are the functors
preserving $\kappa$\+directed colimits and colimits of
$\lambda$\+indexed chains, and taking $\kappa$\+presentable objects
to $\kappa$\+presentable objects.
 The $2$\+cells of $\ACC_{\lambda,\kappa}^\kappa$ are the (arbitrary)
natural transformations.

\begin{thm} \label{conical-lax-limit-theorem}
 Let $\kappa$~be a regular cardinal and $\lambda<\kappa$ be a smaller
infinite cardinal.
 Let $\Gamma$ be a $\kappa$\+small\/ $2$\+category and
$H\:\Gamma\rarrow\ACC_{\lambda,\kappa}^\kappa$ be a\/ $2$\+functor.
 Then the lax limit\/ $\sL$ of the\/ $2$\+functor $H$ (computed in\/
$\TwoCAT$, as per the construction above) belongs to\/
$\ACC_{\lambda,\kappa}^\kappa$.
 For every object $\gamma\in\Gamma$, the natural forgetful/projection
functor $\sL\rarrow H(\gamma)$ belongs to\/
$\ACC_{\lambda,\kappa}^\kappa$.
 An object $S\in\sL$ is $\kappa$\+presentable if and only if,
for every object $\gamma\in\Gamma$, the image $S_\gamma$ of $S$ in
$H(\gamma)$ is $\kappa$\+presentable.
\end{thm}

\begin{proof}
 Similar to the proof of Theorem~\ref{conical-oplax-limit-theorem},
with the directions of some arrows suitably reversed as needed.
\end{proof}

\begin{thm} \label{conical-pseudolimit-theorem}
 Let $\kappa$~be a regular cardinal and $\lambda<\kappa$ be a smaller
infinite cardinal.
 Let $\Gamma$ be a $\kappa$\+small\/ $2$\+category and
$H\:\Gamma\rarrow\ACC_{\lambda,\kappa}^\kappa$ be a\/ $2$\+functor.
 Then the pseudolimit\/ $\sE$ of the\/ $2$\+functor $H$ (computed in\/
$\TwoCAT$, as per the construction above) belongs to\/
$\ACC_{\lambda,\kappa}^\kappa$.
 For every object $\gamma\in\Gamma$, the natural forgetful/projection
functor $\sE\rarrow H(\gamma)$ belongs to\/
$\ACC_{\lambda,\kappa}^\kappa$.
 An object $S\in\sE$ is $\kappa$\+presentable if and only if,
for every object $\gamma\in\Gamma$, the image $S_\gamma$ of $S$ in
$H(\gamma)$ is $\kappa$\+presentable.
\end{thm}

\begin{proof}
 Similar to the proofs of
Theorems~\ref{conical-oplax-limit-theorem}
and~\ref{conical-lax-limit-theorem}, with the only difference that it
is convenient to use the isomorpher construction of
Remark~\ref{isomorpher-remark} instead of the inserter construction
of Theorem~\ref{inserter-theorem}.
 The equifier construction of Theorem~\ref{equifier-theorem} still
needs to be used.
 (Cf.~\cite[Propositions~4.4 and~5.1]{Kel}
and~\cite[Proposition~2.1]{BKPS}.)
\end{proof}

\begin{rem}
 The notions of (op)lax limit and pseudolimit are somewhat relaxed.
 The related strict notion is the \emph{$2$\+limit} of categories.
 $2$\+limits of categories are \emph{not} well-behaved in connection
with accessible categories, generally speaking~\cite[paragraph after
Proposition~5.1.1]{MP}, \cite[Example~2.68]{AR}.
 The well-behaved ones among the (weighted) $2$\+limits are called
\emph{flexible limits} in~\cite{BKPS}.
 Still, the (op)lax limits and pseudolimits are strict enough to be
defined \emph{up to isomorphism of categories} (as per
the constructions above) rather than just up to category equivalence.

 The case of the pseudopullback is instructive.
 Let $\Gamma$ be the following small $2$\+category.
 The $2$\+category $\Gamma$ has three objects $A$, $B$, and $C$,
and two nonidentity $1$\+cells $a\:A\rarrow C$ and $b\:B\rarrow C$.
 There are no nonidentity $2$\+cells in~$\Gamma$.
 Hence a $2$\+functor $H\:\Gamma\rarrow\TwoCAT$ is the same thing as
a triple of categories $\sA$, $\sB$, and $\sC$ together with a pair of
functors $\Theta_\sA\:\sA\rarrow\sC$ and $\Theta_\sB\:\sB\rarrow\sC$,
as in Section~\ref{pseudopullback-secn}.
 Then~\cite[paragraph after Proposition~5.1.1]{MP} explains that
the $2$\+pullbacks, i.~e., the $2$\+limits of $2$\+functors
$H\:\Gamma\rarrow\TwoCAT$, do \emph{not} preserve accessibility
of categories.

 The (op)lax limits and pseudolimits are better behaved and preserve
accessibility, as per the theorems above in this section; but one has
to be careful.
 Looking into these constructions, one can observe that the definition
of the pseudopullback in Section~\ref{pseudopullback-secn} was,
strictly speaking, an abuse of terminology.
 The pseudolimit $\sE$ of a $2$\+functor $H\:\Gamma\rarrow\TwoCAT$ is
the category of all quintuples $(A,B,C,\theta_a,\theta_b)$, where
$A\in\sA$, \,$B\in\sB$, and $C\in\sC$ are three objects and
$\theta_a\:\Theta_\sA(A)\simeq C$, \ $\theta_b\:\Theta_\sB(B)\simeq C$
are two isomorphisms (cf.~\cite[Proposition~3.1]{CR},
\cite[Pseudopullback Theorem~2.2]{RR}).
 The pseudopullback $\sD$ as defined in
Section~\ref{pseudopullback-secn} is \emph{naturally equivalent} to
the pseudolimit $\sE$ of the $2$\+functor $H$, but \emph{not}
isomorphic to it.

 The even more relaxed notion of a limit of categories defined up to
a category equivalence is called
the \emph{bilimit}~\cite[Section~5.1.1]{MP}, \cite[Section~6]{Kel}.
 In the terminology of~\cite[Section~5.1.1]{MP}, the pseudolimits
are called \emph{strong bilimits}.
\end{rem}

\Section{Weighted Pseudolimits}  \label{weighted-secn}

 Let $\Gamma$ be a small $2$\+category and $W\:\Gamma\rarrow
\TwoCat$ be a $2$\+functor (so the category $W(\gamma)$ is small
for every $\gamma\in\Gamma$).
 The $2$\+functor $W$ is called a \emph{weight}.

 Let $H\:\Gamma\rarrow\TwoCAT$ be another $2$\+functor.
 The \emph{weighted pseudolimit} $\{W,H\}_p$
\,\cite[Sections~1\+-2]{BKPS} (called ``indexed pseudolimit'' in
the terminology of~\cite[Sections~2 and~5]{Kel} or ``strong weighted
bilimit'' in the terminology of~\cite[Section~5.1.1]{MP}) can be
simply constructed as the category of $1$\+cells $W\rarrow H$ in
the $2$\+category of pseudonatural transformations
$\Psd[\Gamma,\TwoCAT]$ (mentioned in Section~\ref{2cat-prelim-secn}).
 So $\{W,H\}_p=\Psd[\Gamma,\TwoCAT](W,H)$ \,\cite[formula~(5.5)]{Kel}.

 The strict version of the same construction is
the \emph{weighted\/ $2$\+limit} $\{W,H\}$, which can be defined
as the category of $1$\+cells $W\rarrow H$ in the $2$\+category
of $2$\+natural transformations $[\Gamma,\TwoCAT]$; so
$\{W,H\}=[\Gamma,\TwoCAT](W,H)$ \,\cite[formula~(2.5)]{Kel}.
 It is explained in~\cite[Section~4]{Kel} or~\cite[Section~1]{BKPS}
how to obtain the inserters, equifiers, and isomorphers (iso-inserters)
as particular cases of weighted $2$\+limits.
 Up to category equivalence, they are also particular cases of
weighted pseudolimits.

 Taking $\Gamma$ to be the $2$\+category with a single object,
a single $1$\+cell, and a single $2$\+cell, one obtains
the construction of the diagram category (as in
Theorem~\ref{nonadditive-diagram-theorem}), called the ``cotensor
product'' in~\cite[Section~3]{Kel}, \cite[Section~1]{BKPS},
as the particular case of the weighted $2$\+limit or weighted
pseudolimit.

 Taking $W$ to be the $2$\+functor assigning to every object
$\gamma\in\Gamma$ the category with a single object and a single
morphism, one obtains the construction of the pseudolimit from
Section~\ref{conical-secn} as a particular case of weighted pseudolimit.
 To distinguish them from the more general weighted pseudolimits,
the pseudolimits from Section~\ref{conical-secn} are called
\emph{conical pseudolimits}~\cite[Sections~3 and~5]{Kel},
\cite[Sections~1\+-2]{BKPS}.

 The notation $\ACC_{\lambda,\kappa}^\kappa\subset\ACC_{\lambda,\kappa}
\subset\TwoCAT$ was introduced in Section~\ref{conical-secn}.

\begin{thm} \label{weighted-pseudolimit-theorem}
 Let $\kappa$~be a regular cardinal and $\lambda<\kappa$ be a smaller
infinite cardinal.
 Let $\Gamma$ be a $\kappa$\+small\/ $2$\+category and
$W\:\Gamma\rarrow\TwoCat$ be a $2$\+functor such that the category
$W(\gamma)$ is $\kappa$\+small for every object $\gamma\in\Gamma$.
 Let $H\:\Gamma\rarrow\ACC_{\lambda,\kappa}^\kappa$ be
a\/ $2$\+functor.
 Then the weighted pseudolimit $\{W,H\}_p$ (computed in\/
$\TwoCAT$, as per the construction above) belongs to\/
$\ACC_{\lambda,\kappa}^\kappa$.
\end{thm}

\begin{proof}
 The point is that all weighted pseudolimits can be constructed in terms
of products, inserters, and equifiers~\cite[Proposition~5.2]{Kel},
\cite[Proposition~2.1]{BKPS}; so the assertion follows from
Proposition~\ref{product-proposition}, Theorem~\ref{equifier-theorem},
and Theorem~\ref{inserter-theorem}.
 The same argument applies also to all weighted
bilimits~\cite[Section~6]{Kel} and all flexible weighted
$2$\+limits~\cite[Theorem~4.9 and Remark~7.6]{BKPS}.
\end{proof}

\begin{cor} \label{weighted-pseudolimit-corollary}
 Let $\lambda$ and~$\kappa$ be infinite regular cardinals such that
$\lambda\triangleleft\kappa$ in the sense of\/~\cite[\S2.3]{MP}
or\/~\cite[Definition~2.12]{AR}.
 Let\/ $\Gamma$ be a $\kappa$\+small\/ $2$\+category and
$W\:\Gamma\rarrow\TwoCat$ be a $2$\+functor such that the category
$W(\gamma)$ is $\kappa$\+small for every object $\gamma\in\Gamma$.
 Let $H\:\Gamma\rarrow\TwoCAT$ be a\/ $2$\+functor such that, for
every object $\gamma\in\Gamma$, the category $H(\gamma)$ is
$\lambda$\+accessible, and for every\/ $1$\+cell $a\:\gamma\rarrow
\delta$ in\/ $\Gamma$, the functor $H(a)\:H(\gamma)\rarrow H(\delta)$
preserves $\lambda$\+directed colimits and takes $\kappa$\+presentable
objects to $\kappa$\+presentable objects.
 Then the weighted pseudolimit $\{W,H\}_p$ (computed in\/ $\TwoCAT$,
as per the construction above) is a\/ $\kappa$\+accessible category.
\end{cor}

\begin{proof}
 Follows immediately from
Theorem~\ref{weighted-pseudolimit-theorem}.
\end{proof}

\begin{rem}
 The assertion of Theorem~\ref{weighted-pseudolimit-theorem} captures
many, but not all the aspects of the preceding results in this paper.
 In particular, Theorems~\ref{equifier-theorem}
and~\ref{inserter-theorem} are \emph{not} particular cases of
Theorem~\ref{weighted-pseudolimit-theorem}, if only because
the assumptions of Theorems~\ref{equifier-theorem}\+-%
\ref{inserter-theorem} are more general.
 Indeed, in the assumptions of Theorems~\ref{equifier-theorem}\+-%
\ref{inserter-theorem} the functor $F$ is required to belong to
$\ACC_{\lambda,\kappa}^\kappa$, while the functor $G$ may belong
to the wider $2$\+category $\ACC_{\lambda,\kappa}$.
 In other words, the functor $G$ \emph{need not} take
$\kappa$\+presentable objects to $\kappa$\+presentable objects.
 This subtlety, which was emphasized already in~\cite[Section~3]{Ulm},
manifests itself in the related difference between the formulations
of Theorem~\ref{conical-oplax-limit-theorem}, on the one hand, and
Theorems~\ref{conical-lax-limit-theorem}\+-%
\ref{conical-pseudolimit-theorem}, on the other hand.
 It plays an important role in the application to comodules over
corings worked out in~\cite[Theorem~3.1 and Remark~3.2]{Pflcc} and
in the application to corings in~\cite[Theorem~4.2]{Pcor}.
\end{rem}

\Section{Toy Examples}

 The examples in this section aim to illustrate the main results
of this paper in the context of additive categories, modules categories,
and flat modules, which served as the main motivation for the present
research.
 We refer to the papers~\cite{PS6,Pflcc,Pres,Pcor} for more substantial
applications to flat quasi-coherent sheaves, flat comodules and
contramodules, arbitrary and flat coalgebras and corings,
and flat/injective (co)resolutions.
 This section also serves as a reference source
for~\cite{PS6,Pflcc,Pres,Pcor}, as it contains some results that are
useful as building blocks for the more complicated constructions.

\subsection{Modules and flat modules}
 Let $R$ be an associative ring.
 We denote by $R\Modl$ the abelian category of left $R$\+modules
and by $R\Modl_\fl\subset R\Modl$ the full subcategory of flat
left $R$\+modules.

 The following two propositions are fairly standard.
 
\begin{prop} \label{modules-locally-presentable}
 For any ring $R$ and any regular cardinal~$\kappa$, the category
of $R$\+modules $R\Modl$ is locally $\kappa$\+presentable.
 The $\kappa$\+presentable objects of $R\Modl$ are precisely all
the left $R$\+modules that can be constructed as the cokernel of
a morphism of free left $R$\+modules with less than~$\kappa$
generators.  \qed
\end{prop}

\begin{prop} \label{flat-modules-accessible}
 For any ring $R$ and any regular cardinal~$\kappa$, the category
of flat $R$\+modules $R\Modl_\fl$ is $\kappa$\+accessible.
 All directed colimits exist in $R\Modl_\fl$ and agree with
the ones in $R\Modl$.
 The $\kappa$\+presentable objects of $R\Modl_\fl$ are precisely
all those flat left $R$\+modules that are $\kappa$\+presentable
as objects of $R\Modl$.
\end{prop}

\begin{proof}
 The connection between the present proposition and the previous one
fits into the setting described in
Proposition~\ref{accessible-subcategory}.
 The assertions for $\kappa=\aleph_0$ are corollaries of the classical
Govorov--Lazard theorem~\cite{Gov,Laz} characterizing the flat
$R$\+modules as the directed colimits of finitely generated projective
(or free) $R$\+modules.
 The general case of an arbitrary regular cardinal~$\kappa$ can be
deduced by applying~\cite[Theorem~2.11 and Example~2.13(1)]{AR}.
\end{proof}

 For a version of Proposition~\ref{flat-modules-accessible} for
modules of bounded flat dimension, see~\cite[Corollary~5.2]{Pres}.

\subsection{Diagrams of flat modules}
 The following two corollaries are our ``toy applications'' of
Theorem~\ref{k-linear-diagram-theorem}.

\begin{cor} \label{flat-modules-diagrams}
 Let $k$~be a commutative ring and $R$ be an associative, unital
$k$\+algebra.
 Let $\kappa$~be an uncountable regular cardinal and $\sA$ be
a $\kappa$\+presented $k$\+linear category (in the sense of
Section~\ref{diagram-secn}).
 Then any $k$\+linear functor\/ $\sA\rarrow R\Modl_\fl$ is
a $\kappa$\+directed colimit of $k$\+linear functors\/ $\sA\rarrow
R\Modl_{\fl,<\kappa}$ into the category of $\kappa$\+presentable flat
left $R$\+modules $R\Modl_{\fl,<\kappa}$.
\end{cor}

\begin{proof}
 By Proposition~\ref{flat-modules-accessible} and
Theorem~\ref{k-linear-diagram-theorem} (with $\lambda=\aleph_0$),
the $k$\+linear functor/diagram category $\Fun_k(\sA,R\Modl_\fl)$
is $\kappa$\+accessible, and $\Fun_k(\sA,R\Modl_{\fl,<\kappa})$
is its full subcategory of $\kappa$\+presentable objects.
\end{proof}

\begin{cor} \label{flat-modules-complexes}
 Let $R$ be an associative ring and $\kappa$~be an uncountable
regular cardinal.
 Then any cochain complex of flat $R$\+modules is a $\kappa$\+directed
colimit of complexes of $\kappa$\+presentable flat $R$\+modules.
\end{cor}

\begin{proof}
 This is the particular case of Corollary~\ref{flat-modules-diagrams}
for the ring $k=\boZ$ and the suitable choice of additive category
$\sA$ describing cochain complexes.
 The objects of $\sA$ are the integers $n\in\boZ$, the set of
generating morphisms is the singleton $\Gen(n,m)=\{d_n\}$ for $m=n+1$
and the empty set otherwise, and the set of defining relations is
the singleton $\Rel(n,m)=\{d_{n+1}d_n\}$ for $m=n+2$ and the empty set
otherwise.
\end{proof}

 For a quasi-coherent sheaf, a comodule, and a contramodule version of
Corollary~\ref{flat-modules-complexes},
see~\cite[Theorem~4.1]{PS6} and~\cite[Propositions~3.3 and~10.2]{Pflcc}.

\begin{rem}
 For an uncountable regular cardinal~$\kappa$, the complexes of
$\kappa$\+presentable $R$\+modules are precisely all
the $\kappa$\+presentable objects of the locally finitely
presentable (hence locally $\kappa$\+presentable) abelian category
of complexes of $R$\+modules.
 For $\kappa=\aleph_0$, the finitely presentable objects of
the category of complexes of $R$\+modules are the \emph{bounded}
complexes of finitely presentable $R$\+modules.

 Notice that \emph{not} every complex of flat $R$\+modules is
a directed colimit of bounded complexes of finitely presentable flat
(i.~e., finitely generated projective) $R$\+modules.
 In fact, the directed colimits of bounded complexes of finitely
generated projective $R$\+modules are the \emph{homotopy flat}
complexes of flat $R$\+modules~\cite[Theorem~1.1]{CH}.

 Using the argument from~\cite[proof of
Theorem~2.11\,(iv)\,$\Rightarrow$\,(i)]{AR}
(for $\lambda=\aleph_0$ and $\mu=\kappa$), one can deduce
the assertion that any homotopy flat complex of flat $R$\+modules
is a $\kappa$\+directed colimit of homotopy flat complexes of
$\kappa$\+presentable flat $R$\+modules, for any uncountable
regular cardinal~$\kappa$.
 A quasi-coherent sheaf version of this observation can be found
in~\cite[Theorem~4.5]{PS6}.
\end{rem}

\subsection{Categories of epimorphisms}
 For any category $\sK$, we denote by $\sK^\to$ the category of
morphisms in~$\sK$ (with commutative squares in $\sK$
as morphisms in~$\sK^\to$).
 The following lemma is not difficult.

\begin{lem} \label{arrow-category}
 For any regular cardinal~$\kappa$ and $\kappa$\+accessible
category\/ $\sK$, the category of morphisms\/ $\sK^\to$ is
$\kappa$\+accessible.
 The full subcategory of $\kappa$\+presentable objects in\/ $\sK^\to$
is the category $(\sK_{<\kappa})^\to$ of morphisms of
$\kappa$\+presentable objects in\/~$\sK$.
\end{lem}

\begin{proof}
 One has $\sK^\to=\Fun(\sC,\sK)$ for the obvious finite category~$\sC$
with no nonidentity endomorphisms; so the result
of~\cite[Expos\'e~I, Proposition~8.8.5]{SGA4}, \cite[page~55]{Mey},
or~\cite[Theorem~1.3]{Hen} is applicable.
\end{proof}

 For any category $\sK$, let us denote by $\sK^\epi\subset\sK^\to$
the full subcategory whose objects are all the epimorphisms in~$\sK$.

\begin{lem} \label{loc-pres-epimorphism-category}
 For any regular cardinal~$\kappa$ and any locally $\kappa$\+presentable
abelian category\/ $\sK$, the category of epimorphisms\/ $\sK^\epi$ is
locally $\kappa$\+presentable.
 The full subcategory of $\kappa$\+presentable objects in $\sK^\epi$ is
the category $(\sK_{<\kappa})^\epi$ of epimorphisms between
$\kappa$\+presentable objects in\/~$\sK$.
\end{lem}

\begin{proof}
 Notice first of all that a morphism in $\sK_{<\kappa}$ is
an epimorphism in $\sK_{<\kappa}$ if and only if it is an epimorphism
in~$\sK$ (because the full subcategory $\sK_{<\kappa}$ is closed
under cokernels in~$\sK$ \,\cite[Proposition~1.16]{AR}).
 Furthermore, the full subcategory $\sK^\epi$ is closed under colimits
in the locally presentable abelian category $\sK^\to$; so all colimits
exist in~$\sK^\epi$.
 In view of Lemma~\ref{arrow-category} and according to
Proposition~\ref{accessible-subcategory}, in order to prove the lemma
it suffices to check that any morphism from an object of
$(\sK_{<\kappa})^\to$ to an object of $\sK^\epi$ factorizes through
an object of $(\sK_{<\kappa})^\epi$ in~$\sK^\to$.

 Indeed, consider a commutative square diagram in~$\sK$
$$
 \xymatrix{
  S \ar[r] \ar[d] & K \ar@{->>}[d] \\
  T \ar[r] & L
 }
$$
with an epimorphism $K\twoheadrightarrow L$ and objects
$S$, $T\in\sK_{<\kappa}$.
 Let $M$ be the pullback of the pair of morphisms $K\rarrow L$
and $T\rarrow L$ in~$\sK$; then $M\rarrow T$ is also an epimorphism
(since the category $\sK$ is assumed to be abelian).
$$
 \xymatrix{
  S \ar[r] \ar[d] & M \ar[r] \ar@{->>}[d] & K \ar@{->>}[d] \\
  T \ar@{=}[r] & T \ar[r] & L
 }
$$

 Let $M=\varinjlim_{\xi\in\Xi}U_\xi$ be a representation of $M$
as a $\kappa$\+filtered colimit of $\kappa$\+presentable objects
$U_\xi$ in~$\sK$, and let $V_\xi$ denote the images of
the compositions $U_\xi\rarrow M\rarrow T$.
 The $\kappa$\+filtered colimits are exact functors in~$\sK$
\,\cite[Proposition~1.59]{AR}; hence we have
$T=\varinjlim_{\xi\in\Xi}V_\xi$.
 Since $T\in\sK_{<\kappa}$, it follows that there exists $\xi_0\in\Xi$
such that the morphism $V_{\xi_0}\rarrow T$ is a retraction
(as $V_{\xi_0}\rarrow T$ is a monomorphism by construction, this means
that $V_{\xi_0}\rarrow T$ is actually an isomorphism).
 Hence the composition $U_{\xi_0}\rarrow M\rarrow T$ is an epimorphism.
 Since $S\in\sK_{<\kappa}$, one can choose an index $\xi_1\in\Xi$
together with an arrow $\xi_0\rarrow\xi_1$ in $\Xi$ such that
the morphism $S\rarrow M$ factorizes through the morphism
$U_{\xi_1}\rarrow M$.
 Hence we arrive to the desired factorization
$$
 \xymatrix{
  S \ar[r] \ar[d] & U_{\xi_1} \ar[r] \ar@{->>}[d] & K \ar@{->>}[d] \\
  T \ar@{=}[r] & T \ar[r] & L
 }
$$
through an object $(U_{\xi_1}\to T)\in(\sK_{<\kappa})^\epi$.
\end{proof}

\begin{rem}
 It follows immediately from the first assertion of
Lemma~\ref{loc-pres-epimorphism-category} that the category
$\sK^\mono$ of monomorphisms in $\sK$ is also locally
$\kappa$\+presentable.
 In fact, the categories $\sK^\epi$ and $\sK^\mono$ are naturally
equivalent; the functors of the kernel of an epimorphism and
the cokernel of a monomorphism provide the equivalence.
 However, the direct analogue of the second assertion of
Lemma~\ref{loc-pres-epimorphism-category} \emph{fails} for
monomorphisms (even though the full subcategory $\sK^\mono\subset
\sK^\to$ is closed under $\kappa$\+directed colimits
by~\cite[Proposition~1.59]{AR}).
 In fact, a monomorphism~$i$ in $\sK$ is a $\kappa$\+directed colimit
of monomorphisms between $\kappa$\+presentable objects if and only
if $i$~is an admissible monomorphism in the \emph{maximal locally
$\kappa$\+coherent exact structure on\/~$\sK$}
\,\cite[Corollary~3.3]{Plce}.
 In particular, if $R$ is an associative ring that is not left
coherent, then any monomorphism $i\:N\rarrow M$ from a finitely
generated but not finitely presentable left $R$\+module $N$ to
a finitely presentable left $R$\+module $M$ is \emph{not} a directed
colimit of monomorphisms of finitely presentable modules in $R\Modl$.
\end{rem}

 Given a ring $R$ and a full subcategory $\sL\subset R\Modl$, we
denote by $\sL^\srj\subset\sL^\to$ the full subcategory whose
objects are all the surjective morphisms between objects of~$\sL$.

\begin{lem} \label{flat-adm-epi-category}
 For any associative ring $R$ and any regular cardinal~$\kappa$,
the category of surjective morphisms of flat $R$\+modules
$R\Modl_\fl^\srj$ is $\kappa$\+accessible.
 The $\kappa$\+presentable objects of $R\Modl_\fl^\srj$ are
the surjective morphisms of $\kappa$\+presentable flat $R$\+modules.
\end{lem}

\begin{proof}
 The argument is similar to the proof of
Lemma~\ref{loc-pres-epimorphism-category}.
 In view of Proposition~\ref{flat-modules-accessible},
Lemma~\ref{arrow-category} is applicable to $\sK=R\Modl_\fl$;
so the category of morphisms of flat $R$\+modules $R\Modl_\fl^\to$
is $\kappa$\+accessible and the category of morphisms of
$\kappa$\+presentable flat $R$\+modules $R\Modl_{\fl,<\kappa}^\to$
is the full subcategory of $\kappa$\+presentable objects
in $R\Modl_\fl^\to$.
 According to Proposition~\ref{accessible-subcategory}, in order
to prove the lemma it suffices to check that any morphism from
an object of $R\Modl_{\fl,<\kappa}^\to$ to an object of
$R\Modl_\fl^\srj$ factorizes through an object of
$(R\Modl_{\fl,<\kappa})^\srj$.

 Following the proof of Lemma~\ref{loc-pres-epimorphism-category},
one needs to observe that if $K\twoheadrightarrow L$ is a surjective
morphism of flat $R$\+modules and $T\rarrow L$ is a morphism of
flat $R$\+modules, then the pullback~$M$ (computed in the category
$R\Modl$) is a flat $R$\+module.
 Indeed, the kernel $F$ of the morphism $K\rarrow L$ is a flat
$R$\+module, so the short exact sequence $0\rarrow F\rarrow M
\rarrow T\rarrow0$ shows that $M$ is a flat $R$\+module, too.
 The images $V_\xi$ of the morphisms $U_\xi\rarrow T$ can be taken
in the ambient abelian category $R\Modl$.
 Otherwise, the argument is the same, except that one considers
surjective morphisms in $R\Modl_\fl$ rather than epimorphisms in~$\sK$.
\end{proof}

\begin{rem}
 Alternatively, one can drop the assumption that the category $\sK$
is abelian in Lemma~\ref{loc-pres-epimorphism-category}, requiring it
only to be additive; but assume the cardinal~$\kappa$ to be
uncountable instead.
 Then the resulting assertion can be obtained as a particular case
of Corollary~\ref{pseudopullback-corollary}.
 Consider the category of morphisms $\sA=\sK^\to$, the zero
category $\sB=\{0\}$, and the category $\sC=\sK$.
 Let $\Theta_\sA\:\sA\rarrow\sC$ be the cokernel functor $f\longmapsto
\operatorname{coker}(f)$ and $\Theta_\sB\:\sB\rarrow\sC$ be
the zero functor.
 Then the pseudopullback $\sD$ is the category of epimorphisms
$\sD=\sK^\epi$.
 All the assumptions of Corollary~\ref{pseudopullback-corollary}
(with $\lambda=\aleph_0$) are satisfied; so the corollary tells
that $\sK^\epi$ is $\kappa$\+accessible and provides the desired
description of $\kappa$\+presentable objects.

 Similarly, assuming $\kappa$~to be uncountable, one can deduce
Lemma~\ref{flat-adm-epi-category} from Lemmas~\ref{arrow-category}
and~\ref{loc-pres-epimorphism-category} using
Corollary~\ref{pseudopullback-corollary}.
 Consider the category of $R$\+module epimorphisms $\sA=R\Modl^\epi$,
the category of morphisms of flat $R$\+modules $\sB=R\Modl_\fl^\to$,
and the category of $R$\+module morphisms $\sC=R\Modl^\to$.
 Let $\Theta_\sA\:\sA\rarrow\sC$ and $\Theta_\sB\:\sB\rarrow\sC$
be the natural inclusions.
 Then the pseudopullback $\sD$ is the category of surjective
morphisms of flat $R$\+modules $R\Modl_\fl^\srj$, and
Corollary~\ref{pseudopullback-corollary} is applicable.
\end{rem}

\subsection{Short exact sequences of flat modules}
 Now we can deduce the following three corollaries of
Lemma~\ref{flat-adm-epi-category}.

\begin{cor} \label{short-exact-of-flat-kappa-small-colimits}
 Let $R$ be an associative ring and $\kappa$~be a regular cardinal.
 Then any surjective morphism of $\kappa$\+presentable flat
$R$\+modules is a direct summand of a $\kappa$\+small directed colimit
of surjective morphisms of finitely generated projective $R$\+modules
(in the category $R\Modl_\fl^\to$).
\end{cor}

\begin{proof}
 This follows from Lemma~\ref{flat-adm-epi-category} in view
of~\cite[proof of Theorem~2.11\,(iv)\,$\Rightarrow$\,(i)]{AR} for
$\sK=R\Modl_\fl^\srj$, \ $\lambda=\aleph_0$, and $\mu=\kappa$.
 The Govorov--Lazard characterization of flat modules~\cite{Gov,Laz}
implies that all finitely presentable flat $R$\+modules are projective.
 By Lemma~\ref{flat-adm-epi-category}, the category of surjective
morphisms of flat $R$\+modules is finitely accessible, and
its finitely presentable objects are the surjective morphisms of
finitely generated projective $R$\+modules.
 So all surjective morphisms of flat $R$\+modules are directed
colimits of surjective morphisms of finitely generated projective
$R$\+modules.

 Let $\sA$ denote the set of all $\kappa$\+small directed colimits
of surjective morphisms of finitely generated projective
$R$\+modules.
 Following the argument
in~\cite[proof of Theorem~2.11\,(iv)\,$\Rightarrow$\,(i)]{AR}
and~\cite[Example~2.13(1)]{AR}, all the objects of $R\Modl_\fl^\srj$
are $\kappa$\+directed colimits of objects from~$\sA$.
 Thus all the $\kappa$\+presentable objects of $R\Modl_\fl^\srj$
are direct summands of objects from~$\sA$.
\end{proof}

 The next corollary is a generalization of~\cite[Lemma~4.1]{Pflcc}.

\begin{cor} \label{flat-kappa-coherence}
 Let $R$ be an associative ring and $\kappa$~be a regular cardinal.
 Then the kernel of any surjective morphism of $\kappa$\+presentable
flat $R$\+modules is a $\kappa$\+presentable flat $R$\+module.
\end{cor}

\begin{proof}
 Follows from Corollary~\ref{short-exact-of-flat-kappa-small-colimits},
as the kernel of any surjective morphism of finitely generated
projective $R$\+modules is a finitely generated projective $R$\+module.
 For another proof, see~\cite[Corollary~4.7]{Plce}.
\end{proof}

 Given a ring $R$ and a full subcategory $\sL\subset R\Modl$, let us
denote by $\sL^\ses$ the category of all short exact sequences
in $R\Modl$ with the terms belonging to~$\sL$.

\begin{cor} \label{short-exact-of-flat-modules}
 For any associative ring $R$ and any regular cardinal~$\kappa$,
the category of short exact sequences of flat $R$\+modules
$R\Modl_\fl^\ses$ is $\kappa$\+accessible.
 The full subcategory of $\kappa$\+presentable objects of
$R\Modl_\fl^\ses$ is the category $(R\Modl_{\fl,<\kappa})^\ses$
of all short exact sequences of $\kappa$\+presentable flat
$R$\+modules.
\end{cor}

\begin{proof}
 By Corollary~\ref{flat-kappa-coherence}, the obvious equivalence
of categories $R\Modl_\fl^\srj\simeq R\Modl_\fl^\ses$ identifies
$(R\Modl_{\fl,<\kappa})^\srj$ with $(R\Modl_{\fl,<\kappa})^\ses$.
 This makes the desired assertion a restatement of
Lemma~\ref{flat-adm-epi-category}.
\end{proof}

\subsection{Pure acyclic complexes of flat modules}
 Finally, we can present our ``toy application'' of
Corollary~\ref{pseudopullback-corollary}.
 An acyclic complex of flat $R$\+modules is said to be
\emph{pure acyclic} if its modules of cocycles are flat.

 The following corollary is essentially a weaker version of
the result of~\cite[Theorem~2.4\,(1)\,$\Leftrightarrow$\,(3)]{EG}
or~\cite[Theorem~8.6\,(ii)\,$\Leftrightarrow$\,(iii)]{Neem}.
 Our argument produces it as an application of general
category-theoretic principles.
 See~\cite[Theorem~4.2]{PS6} and~\cite[Corollaries~4.5 and~11.4]{Pflcc}
for a quasi-coherent sheaf, a comodule, and a contramodule version.

\begin{cor}
 Let $R$ be an associative ring and $\kappa$~be an uncountable
regular cardinal.
 Then any pure acyclic complex of flat $R$\+modules is
a $\kappa$\+directed colimit of pure acyclic complexes of
$\kappa$\+presentable flat $R$\+modules.
\end{cor}

\begin{proof}
 The point is that a pure acyclic complex of flat $R$\+modules $F^\bu$
is the same thing as a collection of short exact sequences of flat
$R$\+modules $0\rarrow G^n\rarrow F^n\rarrow H^n\rarrow0$ together
with a collection of isomorphisms $H^n\simeq G^{n+1}$, \,$n\in\boZ$.
 This means that the category of pure acyclic complexes of flat
$R$\+modules can be constructed from the category of short exact
sequences of flat $R$\+modules $R\Modl_\fl^\ses$ using Cartesian
products (as in Section~\ref{product-secn}) and the isomorpher
construction from Remark~\ref{isomorpher-remark}.

 Specifically, put $\sH=\prod_{n\in\boZ}R\Modl_\fl^\ses$ and
$\sG=\prod_{n\in\boZ}R\Modl_\fl$.
 Let $P\:\sH\rarrow\sG$ be the functor taking a collection of
short exact sequences $(0\to G^n\to F^n\to H^n\to0)_{n\in\boZ}$ to
the collection of modules $(H^n)_{n\in\boZ}$, and let
$Q\:\sH\rarrow\sG$ be the functor taking the same collection of short
exact sequences to the collection of modules $(G^{n+1})_{n\in\boZ}$.
 Then the resulting isomorpher category $\sD$ is the category of
pure acyclic complexes of flat $R$\+modules.
 Given the results of Proposition~\ref{flat-modules-accessible}
and Corollary~\ref{short-exact-of-flat-modules}, it follows from
Proposition~\ref{product-proposition} and Remark~\ref{isomorpher-remark}
that the category $\sD$ is $\kappa$\+accessible and the pure acyclic
complexes of $\kappa$\+presentable flat $R$\+modules are precisely all
the $\kappa$\+presentable objects of~$\sD$.
\end{proof}

\end{document}